\documentclass[12pt]{amsart}
\usepackage{amsmath,amssymb,amsthm,amsfonts,mathrsfs,amscd,amsopn}
\usepackage{bbm}
\usepackage{mathscinet}

\title[Quasi-morphisms and symplectic quasi-states]{Quasi-morphisms and symplectic quasi-states for convex symplectic manifolds}
\author{Sergei Lanzat}
\address{Department of Mathematics, Technion -- Israel
Institute of Technology, Haifa 32000, Israel}
\email{serjl@tx.technion.ac.il}

\renewcommand{\(}{\left(}
\renewcommand{\)}{\right)}

\newcommand{\fr}[1]{\mbox{{$\mathfrak{#1}$}}}
\newcommand{\FAT}[1]{\mbox{{$\mathbb{#1}$}}}
\newcommand{\Fat}[1]{\mbox{{$\scriptstyle\mathbb{#1}$}}}
\newcommand{\CL}[1]{\mbox{{$\mathcal{#1}$}}}
\newcommand{\Cl}[1]{\mbox{{$\scriptstyle\mathcal{#1}$}}}
\newcommand{\KL}[1]{\mbox{{$\mathscr{#1}$}}}
\newcommand{\Kl}[1]{\mbox{{$\scriptstyle\mathscr{#1}$}}}
\renewcommand{\til}[1]{\widetilde{#1}}
\renewcommand{\hat}[1]{\widehat{#1}}
\newcommand{\hCL}[1]{\mbox{{$\mathcal{\hat{#1}}$}}}

\newcommand{\hn}[1]{\mbox{$\|{#1}\|_{L^{(1,\infty)}}$}}

\newcommand{\ZZ}{\FAT{Z}}
\newcommand{\NN}{\FAT{N}}
\newcommand{\FF}{\FAT{F}}

\newcommand{\DD}{\FAT{D}}
\newcommand{\QQ}{\FAT{Q}}
\newcommand{\JJ}{\FAT{J}}
\newcommand{\KK}{\FAT{K}}
\newcommand{\PP}{\FAT{P}}
\newcommand{\BB}{\FAT{B}}
\newcommand{\RR}{\FAT{R}}
\newcommand{\TT}{\FAT{T}}
\newcommand{\CC}{\FAT{C}}
\newcommand{\SSS}{\FAT{S}}
\newcommand{\zz}{\Fat{Z}}
\newcommand{\kk}{\Fat{K}}

\newcommand{\ff}{\Fat{F}}

\newcommand{\dd}{\Fat{D}}

\newcommand{\pp}{\Fat{P}}
\newcommand{\rr}{\Fat{R}}
\newcommand{\cc}{\Fat{C}}
\newcommand{\sss}{\Fat{S}}

\newcommand{\qm}{\mathbbm{q}}
\newcommand{\pqm}{\mathbbm{q}_{\hspace{0.1mm}\mathrm{p}}}
\newcommand{\qs}{\zeta}
\newcommand{\pqs}{\zeta_\mathrm{p}}

\newcommand{\IFF}{\Leftrightarrow}
\newcommand{\then}{\Rightarrow}

\newcommand{\cl}[1]{\overline{#1}}
\newcommand{\minus}{\smallsetminus}
\newcommand{\e}{\epsilon}
\newcommand{\ve}{\varepsilon}

\newcommand{\la}{\langle}
\newcommand{\ra}{\rangle}

\newcommand{\x}{\mathbbm{x}}

\newcommand{\PH}[1]{\CL{P}_#1}

\newcommand{\tPH}[1]{\til{\CL{P}_#1}}

\newcommand{\del}{\partial}

\newcommand{\tHam}{\widetilde{\hbox{Ham}}}

\newcommand{\be}{\begin{itemize}}
\newcommand{\ee}{\end{itemize}}

\newcommand{\bd}{\begin{description}}
\newcommand{\ed}{\end{description}}

\newcommand{\beq}{\begin{equation}}
\newcommand{\eeq}{\end{equation}}

\newcommand{\beqn}{\begin{equation}\nonumber}

\newcommand{\bea}{\begin{equation}\begin{aligned}}
\newcommand{\eea}{\end{aligned}\end{equation}}

\newcommand{\bean}{\begin{equation}\nonumber\begin{aligned}}

\DeclareMathOperator{\End}{End}
\DeclareMathOperator{\Span}{Span}

\DeclareMathOperator{\id}{id}

\DeclareMathOperator{\Spec}{Spec}

\DeclareMathOperator{\Ham}{Ham}
\DeclareMathOperator{\Symp}{Symp}

\newtheorem{thm}{Theorem}[section]
\newtheorem{thm*}{Theorem}
\newtheorem{lem}[thm]{Lemma}
\newtheorem{lem*}[thm*]{Lemma}
\newtheorem{prop}[thm]{Proposition}
\newtheorem{cor}[thm]{Corollary}
\newtheorem{prop*}[thm*]{Proposition}
\newtheorem{cor*}[thm*]{Corollary}

\newtheorem{defn}[thm]{Definition}
\newtheorem{defn*}[thm*]{Definition}
\newtheorem{thm-defn}[thm]{Theorem-Definition}

\theoremstyle{remark}
\newtheorem{rem}[thm]{Remark}
\newtheorem{exs}[thm]{Examples}
\newtheorem{ex}[thm]{Example}

\begin{document}

 \begin{abstract}
We use quantum and Floer homology to construct (partial) quasi-morphisms on the universal cover $\tHam_c(M, \omega)$ of the group of compactly supported Hamiltonian diffeomorphisms for a certain class of non-closed strongly semi-positive symplectic manifolds $(M,\omega)$. This leads to a construction of (partial) symplectic quasi-states on the space $C_{cc}(M)$ of continuous  functions on $M$ that are constant near infinity. The work extends the results by  Entov and  Polterovich which apply in the closed case.
 \end{abstract}


\thanks{Partially supported by the ISF grant 881/06.}

\maketitle
\section{Introduction}
\subsection{Overview}

Let $(M,\omega)$ be a  connected symplectic manifold (possibly with boundary).
Let $\Ham_c(M,\omega)$ be the group of Hamiltonian diffeomorphisms of $(M,\omega)$
generated by (time-dependent) Hamiltonians with compact support that lies in the interior of $M$
and let $\tHam_c(M,\omega)$ be  its universal cover.
Let $C_c (M)$ (resp. $C^\infty_c (M)$) be the space of continuous (resp. Poisson-Lie algebra of smooth)
functions on $M$ with compact support that lies in the interior of $M$. In the case of a closed $M$ we will write $\Ham (M,\omega)$, $\tHam(M,\omega)$, $C(M)$ and $C^\infty(M)$ without the lower indices.

In the case of a closed $M$, Entov and  Polterovich
\cite{E-P1}, \cite{E-P2}, \cite{E-P4} (see also \cite{Ostr-qmm},
\cite{U}, \cite{Usher2} for subsequent important developments) used
quantum and  Floer homology to construct certain ``almost
homomorphisms" $\mu: \tHam (M,\omega)\to \RR$ (in some
cases this function descends to a function on $\Ham (M,\omega)$)
with a number of remarkable properties.
In particular, $\mu$ has the following {\it \textbf{Calabi
property}}: its restriction to the subgroup $\til{\Ham}_c (U)\subset
\til{\Ham} (M,\omega)$ for any sufficiently small open subset
$U\subset M$ is the classical Calabi homomorphism, where
$\tHam_c(U)$ is the subgroup of elements generated by Hamiltonians
$H\in C^\infty(\SSS^1\times M)$ with
$\mathrm{supp}(H(t,\cdot))\subset U$ for all $t\in\SSS^1$.

In addition,  Entov and  Polterovich associate to each such ``almost homomorphism" an ``almost linear" functional $\zeta:
C(M)\to\RR$, also with a number of remarkable properties.

The exact meaning of the terms ``almost homomorphism" and ``almost
linear functional" depends on the algebraic structure of the quantum
homology $QH_* (M)$ of $M$. In a somewhat unprecise way this
dependence can be formulated as follows: if the algebra $QH_* (M)$
admits a field as a direct summand (in the category of algebras over
a certain base field), then the ``almost homomorphism" is a {\it
\textbf{homogeneous quasi-morphism}}. Otherwise, for a general
$M$, one gets a function  with weaker properties -- it is
called a {\it \textbf{ partial quasi-morphism}}. Accordingly, in the former
case the ``almost linear functional"  has stronger
properties and is called a {\it \textbf{symplectic quasi-state}} (see the
definition in \cite{E-P2} and Section \ref{susection: quasi-states} below), while in the latter case it has weaker
properties and is called a {\it \textbf{partial symplectic quasi-state}}.

The existence of (partial) quasi-morphisms and (partial) quasi-states
(constructed by means of the quantum and Floer homology) has various interesting applications -- see e.g. \cite{E-P1}, \cite{E-P2},
\cite{E-P3}, \cite{E-P5}, \cite{E-P6}, \cite{E-P-P}, \cite{E-P-R}, \cite{E-P-Z}, \cite{B-I-P}, \cite{Buh-E-P}.
We will discuss these applications further in the paper.

\subsection{Constructions and main results}
The goal of this paper is to construct (partial) quasi-morphisms and (partial) quasi-states for {\it non-closed} (strongly semi-positive -- see Definition~\ref{defn: semi-positive manifolds}) symplectic manifolds. Let us first give exact definitions of these objects.
\begin{defn}
A function $\qm : G \to
{\RR}$ on a group $G$ is called a {\it \textbf{quasi-morphism}} if there
exists a constant $C \geq 0$ such that
$$|\qm(g_1g_2) - \qm(g_1) - \qm(g_2) | \leq C, \ \ {\rm for \ every
\ } g_1,g_2 \in G.$$
The minimal constant $C$ in the above inequality is called the {\it\textbf{defect}} of $\qm$ and will be denoted by $D(\qm)$. A quasi-morphism $\qm$ is called {\it \textbf{homogeneous}} if $\qm(g^n) = n \qm(g)$ for all $g\in G$ and $n \in {\ZZ}$. Any quasi-morphism $\qm$ can be homogenized yielding a homogeneous quasi-morphism $\hat{\qm} (g):= \lim_{n\to +\infty} \qm(g^n)/n$, which, in general, may be a homomorphism.
\end{defn}

\begin{defn}\label{def:fragmentation length}
Let $U\subset M$ be a displaceable open set. By Banyaga's fragmentation lemma (see \cite{B}), each $\til{\phi}\in\tHam_c(M,\omega)$ can be represented as a product of elements of the form $\til{\phi}\til{\theta}\til{\phi}^{-1}$, with $\til{\theta}\in\tHam_c(U)$. The \textbf{fragmentation length} $\left\|\til{\phi}\right\|_U$ of $\til{\phi}$ is the minimal number of factors in such a product.
\end{defn}

\begin{defn}
A function $\pqm:\tHam_c(M,\omega)\to\RR$ is called a \textbf{partial quasi-morphism} if it satisfies following properties:

\item{{\bf(Controlled ~quasi-additivity)}}\ Given a displaceable open set $U\subset M$, there exists a positive constant $R$, depending only on $U$, so that
    $$
    \left|\pqm\(\til{\phi}\til{\psi}\)-
    \pqm\(\til{\phi}\)-\pqm
    \(\til{\psi}\)\right|\leq R\cdot\min\left\{\left\|\til{\phi}\right\|_U, \left\|\til{\psi}\right\|_U\right\}
    $$
    for any $\til{\phi}, \til{\psi}\in\tHam_c(M,\omega)$.
\item{{\bf(Semi-homogeneity)}}\ $\pqm\(\til{\phi}^m\)=m\pqm\(\til{\phi}\)$ for any $\til{\phi}\in\tHam_c(M,\omega)$ and any $m\in\ZZ_{\geq 0}$.
\end{defn}
Denote by $C_{cc}(M)$ the space of all continuous functions $F$ on $M$ endowed with the uniform norm, such that $F\in C_{cc}(M)$  if and only if there exists  $\fr{n}_F\in\RR$, for which $F-\fr{n}_F\in C_c(M)$. Denote also by $ C^{\infty}_{cc}(M)$ the subspace of $ C^{\infty}$-smooth functions inside $C_{cc}(M)$. The following definitions extend the notion of a (partial) symplectic quasi-state to the case of a non-closed $M$.

\begin{defn}\label{Def:symplectic quasi-state}
A (not necessarily linear) functional $\qs: C_{cc}(M)\to\RR$  is called a \textbf{symplectic quasi-state} if it satisfies the following properties:

\item{{\bf(Strong quasi-linearity)}}\ $\qs(\lambda F+G)=\lambda\qs(F)+\qs(G)$ for all Poisson commuting smooth functions $F,G$, i.e. $\{F,G\}=0$ and for all $\lambda\in\RR$.
\item{{\bf(Monotonicity)}}\ $\qs(F)\leq\qs(G)$ for $F\leq G$.
\item{{\bf(Normalization)}}\ $\qs(1)=1$.

\item{{\bf(Vanishing)}}\ $\qs(F)=\fr{n}_F$, provided $\mathrm{supp}(F-\fr{n}_F)$ is  displaceable.
\item{{\bf(Symplectic invariance)}}\ $\qs(F)=\qs(F\circ \phi)$ for $\phi\in\Symp^0_c (M,\omega)$, where $\Symp^0_c(M,\omega)$ denotes the identity component of the group of compactly supported symplectomorphisms of $(M,\omega)$,  whose supports  lie in the interior of $M$.
\end{defn}

\begin{defn}\label{Def:symplectic partial quasi-state}
Let $\pqs: C_{cc}(M)\to \RR$ be a functional, which satisfies monotonicity, normalization, vanishing and symplectic invariance axioms from above. Assume that $\pqs$ has two additional properties:

\item{{\bf(Partial additivity)}}\ If $F, G\in  C_{cc}^{\infty}(M),\ \{F, G\}=0$ and the support of  $G$ is  displaceable, then $\pqs(F + G)=\pqs(F)+\fr{n}_G$.
\item{{\bf(Semi-homogeneity)}}\ $\pqs(\lambda\cdot F)=\lambda\cdot\pqs(F)$ for any $F$ and any $\lambda\in\RR_{\geq 0}$.
We call such a function $\pqs$ a \textbf{partial symplectic quasi-state}.
\end{defn}

Our construction of (partial) quasi-morphisms and quasi-states for non-closed manifolds will follow the general scheme of the Entov-Polterovich construction in the closed case.
In order to implement that scheme in the non-closed case one needs to overcome a number of technical difficulties, which we discuss below and, more importantly, a serious conceptual difficulty: namely, {\sl a priori, it may well happen that the resulting $\qm: \tHam_c (M,\omega)\to
\RR$ will turn out to be just a homomorphism (for instance, a scalar
multiple of the Calabi homomorphism) and the corresponding $\qs: C_{cc}(M)\to\RR$ will be just a linear functional}. (In the
closed case, this cannot happen due to the famous Banyaga's theorem
\cite{B} which implies that there are non-trivial homomorphisms on
$\tHam (M,\omega)$ and the Calabi property of $\mu$ which
guarantees that $\mu$ cannot be just zero.) In the paper we overcome
the latter difficulty for a wide class of non-closed manifolds
$(M,\omega)$ by using the results of Biran and Cornea \cite{B-C2} on the pearl (Lagrangian quantum) homology.

Namely, let $X$ be a smooth closed connected manifold of dimension $2n$ for $n>1$. Assume that $X$  is either a homology sphere over $\ZZ_2$ or the singular homology algebra (with respect to the intersection product) $H_*(X;\ZZ_2)$ is generated as a ring by
$H_{2n-1}(X;\ZZ_2)$. Examples of $X$ such that $H_*(X;\ZZ_2)$ is
generated as a ring by $H_{2n-1}(X;\ZZ_2)$ are direct products of
smooth closed 2-dimensional surfaces (orientable or not),
homology projective spaces over $\ZZ_2$, direct products or connected sums of the above manifolds. We take $(M, L,\omega)$ to be the one-point symplectic blow-up $(\til{\DD^*X}, \til{X},\omega_{\delta})$ of the cotangent closed disk bundle $\DD^*X$ relative to the Lagrangian zero section $X$, i.e. the symplectic blow-up  of size $\delta$ that corresponds to a symplectic embedding of pairs $\displaystyle (\BB^{2\dim X}(\delta), \BB^{\dim X}(\delta))\hookrightarrow(\DD^*X, X)$. Here $\DD^*X$
denotes a cotangent closed disk bundle of $X$, $\BB^{2\dim
X}(\delta)$ is the standard closed ball with the center at zero and
radius $(\frac{\delta}{\pi})^{1/2}$ in $\CC^{\dim X}$ and $\BB^{\dim
X}(\delta)$ is its real part in $\RR^{\dim X}=Re\, \CC^{\dim X}$.

\begin{thm}\label{thm:Main theorem on qm}
Let  $(M, L, \omega)$ be as above. Then there exist  a homogeneous quasi-morphism $\qm: \tHam_c (M,\omega)\to\RR$, which is not a homomorphism and, accordingly, the corresponding quasi-state $\qs: C_{cc} (M)\to\RR$, which is a non-linear functional.
\end{thm}
\noindent
Moreover, we have

\begin{thm}\label{thm:restriction of qm to hamc}
The  homogeneous quasi-morphism $\qm:\tHam_c(M,\omega)\to\RR$ in Theorem~\ref{thm:Main theorem on qm}  descends to $\Ham_c(M,\omega)$.
\end{thm}

For a wider class of non-closed symplectic manifolds $(M,\omega)$
the same method allows to construct a {\it partial} quasi-morphism $\pqm: \tHam_c (M,\omega)\to\RR$, which, a priori, may be a homomorphism (or equal to $\qm$ for $(M,\omega)$ as in Theorem \ref{thm:Main theorem on qm}), see Theorem~\ref{thm: partial quasimorphism r} and Section~\ref{section: qms for exhaustion}. This class contains  convex compact strongly semi-positive symplectic manifolds (see Sections~\ref{subsection: convex symp mnflds}--\ref{subsection: QH} for definitions)  and certain open convex symplectic manifolds, see Section~\ref{section: qms for exhaustion} and  Corollary~\ref{cor: pqm for exhaustion}. In particular, the cotangent bundle $(T^*X,\omega_{can})$ over any closed connected smooth manifold $X$,  an open Stein manifold $(M, J, f)$ (see Example~\ref{example: convex symplectic manifolds}(3)) and  $(M,\omega)$ from Theorem \ref{thm:Main theorem on qm} belong to this class. The next proposition shows that under certain assumptions the partial quasi-morphism $\pqm$ is not a  homomorphism. Recall that a symplectic manifold $(M,\omega)$ is called \textit{\textbf{weakly exact}} if  $[\omega]|_{\pi_2(M)}=0$.

\begin{prop}\label{prop:nontrivial pqs for weakly exact mnflds}
The partial quasi-morphism $\pqm: \tHam_c (M,\omega)\to\RR$ is not a homogeneous quasi-morphism (and, in particular, not a homomorphism)  and, accordingly, the corresponding partial quasi-state $\pqs: C_{cc} (M)\to\RR$ is not a symplectic quasi-state  (and, in particular,  a non-linear functional) under either of the following (mutually non-excluding) assumptions:

\noindent
{\rm {\bf (A)}}\; $(M,\omega)$ is a convex compact weakly exact symplectic manifold  admitting a closed Lagrangian submanifold $L\subset M\minus\del M$, such that
\be
\item[1.]the inclusion map $L\hookrightarrow M$ induces an injection $\pi_1(L)\hookrightarrow\pi_1(M)$,
\item[2.] $L$ admits a Riemannian metric with no non-constant contractible closed geodesics,
\ee
\noindent
{\rm {\bf (B)}}\; $(M,\omega)=(T^*X,\omega_{can})$, where $X$ is any closed connected smooth manifold  admitting  a non-vanishing closed 1-form.

\noindent
If $(M,\omega)$ satisfies {\rm {\bf (A)}} or {\rm {\bf (B)}} then $\pqm$ descends to $\Ham_c(M,\omega)$.
\end{prop}

Let us note that the claim in the case {\rm {\bf (B)}} follows from the results of  Monzner, Vichery and Zapolsky \cite{M-V-Z}, who use a different construction of the same flavor to construct partial quasi-morphisms (which are not homogeneous quasi-morphisms) on $\Ham_c (T^*X,\omega_{can})$ for {\it any} closed connected smooth manifold $X$ -- the claim follows from the comparison of the latter partial quasi-morphisms and $\pqm$.

Accordingly, for such symplectic manifolds, one can use the obtained
(partial) quasi-morphisms and (partial) quasi-states to get numerous
applications similar to the ones in the closed case -- see
Section~\ref{section: application}.

Now let us briefly outline our construction and discuss the
above-mentioned technical difficulties and the ways to overcome
them.

First, since the original construction for closed manifolds involves
moduli spaces of pseudo-holomorphic curves, we need to impose some
restriction on the behavior of $M$ at infinity (or near the
boundary) in order to guarantee that the relevant moduli spaces are
compact. Conditions of this sort are well-known in symplectic
topology and we will use one of them: namely, we will work with strongly semi-positive compact symplectic manifolds with boundary, called {\it
\textbf{convex symplectic manifolds}} -- see
Section~\ref{subsection: convex symp mnflds} for a precise
definition. Examples of convex symplectic manifolds include a closed
symplectic ball, a cotangent closed disk bundle over a closed
manifold, as well as their blow-ups at interior points.

Second, we need to set up an appropriate version of the quantum
homology. Note that, unlike in the closed case, in the case of a
manifold with boundary there are absolute and relative to the
boundary singular homologies and three different versions of the
(classical) intersection product -- see Definition~\ref{defn:
Intersection products}. Accordingly, a straight-forward
generalization of the construction in the closed case yields
relative and absolute quantum homologies $QH^\pm$ of our symplectic
manifold and three versions of the quantum products -- see
Definition~\ref{defn: quantum products}.

Third, we need to set up the Hamiltonian Floer homology to work with
compactly supported (time-dependent time-periodic) Hamiltonians.
Modifying the constructions of  Frauenfelder and Schlenk \cite{F-S},
we get relative and absolute Hamiltonian Floer homologies $FH^\pm$
associated to the same $H$. Both homologies are equipped with a
so-called pair-of-pants product and a filtration induced by the
action functional. The Piunikhin-Salamon-Schwarz construction
\cite{PSS} generalizes from closed to convex symplectic manifolds
and yields canonical ring (PSS) isomorphisms between $FH^\pm$ and
$QH^\pm$. Hence, for a non-zero quantum homology class $a^\pm\in
QH^\pm$ and a Hamiltonian $H$ as above, we get two different
versions of {\it \textbf{spectral numbers}}: $c^\pm (a^\pm, H)$.
They are obtained by sending $a^\pm$ by the PSS isomorphisms to
$FH^\pm$ and measuring its filtration level there. As in the closed
case, the spectral numbers depend only on the element
$\til{\phi}_H\in\tHam_c(M,\omega)$ generated by $H$. Thus we get the
spectral numbers $c^\pm(a^\pm, \til{\phi}_H)$ -- see
Section~\ref{subsection: spectral invariants}.

Next, as in the closed case, in order to build quasi-morphisms from
the spectral numbers we need to have idempotents in the quantum
homology. Assume that idempotents $\e^\pm\in QH^\pm$ are the unities in field summands in $QH^\pm$ (a condition of this sort is
absolutely crucial and appears in the closed case as well) and
satisfy one more assumption related to the Poincar\'{e}-Lefschetz
duality in the quantum homology (this additional condition is not
needed in the closed case, roughly speaking, because the
intersection product of \emph{absolute} homology classes is
non-degenerate). Note that, unlike in the closed case, where only
one idempotent is needed,  we need \emph{two} idempotents $\e^\pm$. The reason for this is that the Poincar\'{e}-Lefschetz isomorphism switches between different homology theories -- absolute and relative. Then we show that the functions $\til{\phi}\mapsto
c^\pm(\e^\pm,\til{\phi})$ are quasi-morphisms on $\tHam_c (M,\omega)$
with the same homogenization $\qm$, see Theorem~\ref{thm: quasi-morphism q on hamtil} and Corollary~\ref{cor: quasi-morphism r on hamtil}. This homogeneous quasi-morphism, shifted by the Calabi homomorphism (which is again a homogeneous quasi-morphism) has the Calabi property, see Corollary~\ref{cor: homogeneous quasi-morphism tilr on hamtil}. As in the closed case, the constructed homogeneous quasi-morphism defines a quasi-state, see Theorem~\ref{thm: quasi-state for QM12 manflds}. Recall that in the closed case a quasi-state, by definition, has to be equal $1$ on $1$. In fact, and this is an additional technical problem appearing in the case of a non-closed $M$, the definition of a quasi-state has to be modified since the constant function $1$ does not lie in $C^\infty_c (M)$ --  the ``Lie algebra" of the ``infinite-dimensional Lie group" $\Ham_c (M,\omega)$. We overcome the difficulty by enlarging the space $C^\infty_c(M)$ to the space $C^\infty_{cc}(M)$.

Finally, and this was already discussed above, we show that the
resulting $\qm$ is not a homomorphism (and the corresponding $\qs$ is, accordingly, non-linear). In fact, in the cases where we can
show  this, we can also show, using a computation involving the
Seidel homomorphism \cite{Seidel} that $\qm$ descends to $\Ham
(M,\omega)$.

\subsection{Organization of the paper.}
The paper is organized in the following way. In Section 2 we recall the definition of (absolute and relative) quantum and Floer
homologies of strongly semi-positive convex compact symplectic manifolds and define absolute and relative analogues  of the spectral invariants on the group of compactly supported Hamiltonian diffeomorphisms. In Section 3 we prove that under appropriate algebraic conditions on the absolute and relative quantum homology algebras of a strongly semi-positive compact convex symplectic manifold $M$, the universal cover of the group of its compactly supported Hamiltonian diffeomorphisms admits a  real-valued homogeneous quasi-morphism. By a ``linearization" of the quasi-morphism we define a  symplectic quasi-state on the space of smooth functions on $M$, which are constant near the boundary $\del M$. Using Biran-Cornea Lagrangian pearl homology, we construct examples of manifolds, for which  the quasi-state is non-linear functional and hence, the quasi-morphism is not a homomorphism. Modifying the Seidel homomorphism construction, we prove that for our examples the quasi-morphism descends  to the group of compactly supported Hamiltonian diffeomorphisms of $M$. Next, we show that for a certain class of weakly exact  compact convex symplectic manifolds the above construction leads to a non-trivial partial quasi-morphism and a partial quasi-state. Finally, we discuss (partial) quasi-morphisms and (partial) quasi-states for non-compact convex manifolds. In Section 4 we discuss several applications in the spirit of Entov-Polterovich  to that class of  symplectic manifolds.

\subsection*{Acknowledgement.}
This paper is a part of  author's Ph.D. thesis, being carried out
under the supervision of Professor Michael Entov and Professor
Michael Polyak, at Technion -- Israel Institute of Technology. I
would like to thank Michael Entov, who introduced me to
this subject, guided and helped me a lot while I was working on this
paper. I am grateful to Leonid Polterovich for important comments
and  advices  that have considerably improved the text. I would like
to thank Felix Schlenk for his valuable mathematical and linguistic
remarks. I am beholden to Frol Zapolsky for stimulating discussions and for explanations concerning the vanishing of  spectral invariants on non-positive compactly supported functions on a cotangent bundle that admits a non-vanishing closed section. Finally, I am grateful to Michael Polyak for his valuable suggestions and comments in the course of my work on this paper.

\section{Symplectic preliminaries.}
In this section we recall the definitions of quantum and Floer homologies of strongly semi-positive convex compact symplectic manifolds. The material is well-known and standard -- for more details see \cite[Chapters 1, 2]{Lan} for more details.
\subsection{Convex symplectic manifolds.}\label{subsection: convex symp mnflds}
 \begin{defn}\label{Def:Convex symplectic manifolds}( \cite{EG},  \cite{F-S}, \cite{McDuff}, \cite{MS3})
\mbox{}
Consider a 2n-dimensional compact symplectic manifold $(M,\omega)$ with a non-empty boundary $\partial M$. The boundary  $\partial M$  is called \textbf{convex} if there exists a Liouville vector field $X$ (i.e. $\CL{L}_X\omega=d\iota_X\omega=\omega$), which is defined in the neighborhood of $\partial M$ and which is everywhere transverse to $\partial M$, pointing outwards. A compact symplectic manifold $(M,\omega)$ with non-empty boundary $\partial M$ is called \textbf{convex} if $\partial M$ is convex.
\end{defn}

\begin{exs}\label{example: convex symplectic manifolds}
\mbox{}
\item{(1)}\,The standard closed $r$-balls $(\BB^{2n}(r),\omega_0)$ in $\RR^{2n}$ are convex.
\item{(2)}\,Cotangent closed $r$-ball bundles $(\DD_r^*X,\omega_{can})$ are convex.
\item{(3)}\,Stein domains: if $(M,J,f)$ is a Stein manifold, where $(M,J)$ is an open complex manifold and $f:M\to\RR$ is a smooth exhausting plurisubharmonic function, such that  the critical points of $f$ are in, say, $\{f<1\}$, then $M_i=\{f\leq i\},\,i\in\NN$, are exact convex compact symplectic manifolds w.r.t the symplectic form $\omega_f:=-d(df\circ J)$.

\end{exs}

\subsection{Quantum homology of compact convex symplectic manifolds.}\label{subsection: QH}
Let the base field $\FF$  be either $\ZZ_2$ or $\CC$.

Recall  the definition of the intersection products for a manifold with boundary.
\begin{defn}\label{defn: Intersection products}
\mbox\\
Homomorphisms
\bean\label{equation: intersection products in homology}
&\bullet_1:H_i(M;\FF)\otimes H_j(M;\FF)\to H_{i+j-2n}(M;\FF)\\
&\bullet_2:H_i(M;\FF)\otimes H_j(M,\partial M;\FF)\to H_{i+j-2n}(M;\FF)\\
&\bullet_3:H_i(M,\partial M;\FF)\otimes H_j(M,\partial M;\FF)\to H_{i+j-2n}(M,\partial M;\FF)
\eea
given by
\bean
&a\bullet_1 b:=\mathrm{PLD}_2\(\mathrm{PLD}_2^{-1}(b)\cup \mathrm{PLD}_2^{-1}(a)\)\\
&a\bullet_2 b:=\mathrm{PLD}_2\(\mathrm{PLD}_2^{-1}(b)\cup \mathrm{PLD}_1^{-1}(a)\)\\
&a\bullet_3 b:=\mathrm{PLD}_1\(\mathrm{PLD}_1^{-1}(b)\cup \mathrm{PLD}_1^{-1}(a)\)
\eea
are called the intersection products in homology.\\ Here,
$
H^j(M;\FF)\overset{\mathrm{PLD}_1}{\longrightarrow} H_{2n -j}(M,\partial M;\FF)$, $H^j(M,\partial M;\FF)\overset{\mathrm{PLD}_2}{\longrightarrow} H_{2n -j}(M;\FF)$
are the Poincar\'{e}-Lefschetz  isomorphisms given by $\mathrm{PLD}_i(\alpha):=\alpha\cap[M,\partial M]$, $i=1,2$, where $[M,\partial M]$ is the relative fundamental class, i.e. the positive generator of $H_{2n}(M,\partial M;\FF)\cong\FF.$
\end{defn}

Next, denote by $H_2^S(M)$ the image of the Hurewicz homomorphism $\pi_2(M)\to H_2(M,\ZZ)$. The homomorphisms  $c_1: H_2^S(M)\to\ZZ$ and $\omega: H_2^S(M)\to\RR$ are given by $c_1(A):=c_1(TM,\omega)(A)$ and $\omega(A)=[\omega](A)$ respectively.

\begin{defn}\label{defn: semi-positive manifolds}(\cite{H-S}, \cite{McDuff})
\mbox{}
A  symplectic $2n$-manifold $(M,\omega)$ is called \textbf{strongly semi-positive}, if $\omega(A)\leq 0$ for any $A\in H_2^S(M)$ with $2-n\leq c_1(A)<0$.
\end{defn}

Let  $(M,\omega)$ be a strongly  semi-positive convex compact symplectic manifold. Consider the space $\CL{J}(M,\partial M,\omega)$ of  $\omega$-compatible almost complex structures on $M$ that are \textbf{\emph{adapted to the boundary}}, i.e.  $J\in\CL{J}(M,\partial M,\omega)$ iff $J$ is $\omega$-compatible and  for all $x\in\partial M$ and for all $v\in T_x\partial M$ we have
$$
J(x)X(x)\in T_x\partial M,\;\;\;\omega(v,J(x)X(x))=0,\;\;\;\omega(X(x),J(x)X(x))=1.
$$
The  space $\CL{J}(M,\partial M,\omega)$  is non-empty and connected, and  all $J$-holomorphic curves  lie  in the complement of some open neighborhood of $\partial M$ for any almost complex structure $J\in\CL{J}(M,\partial M,\omega)$ -- see \cite[ Section 9.2]{MS3} and \cite[Section 1.1.3]{Lan} for more details.
Let $m\geq 3$ be a natural number, $p\in\{0,1,\ldots,m\}$ and $A\in H_2^S(M)$. Fix distinct marked points $(z_1,\ldots,z_m)\in\(\SSS^2\)^m$ and a generic $J\in\CL{J}(M,\partial M,\omega)$. Define the genus zero Gromov-Witten invariant $GW_{A,p,m}$  relative to the boundary as the $m$-linear map over $\FF$ $$GW_{A,p,m}: H_*(M;\FF)^{\times p}\times H_*(M,\partial M;\FF)^{\times (m-p)}\to\FF,$$
which counts the following geometric configurations.\\
If $\FF=\ZZ_2$, choose  smooth cycles $f_i: V_i\to M$  representing $a_i\in H_*(M;\ZZ_2)$ for every $i=1,\ldots, p$, and choose relative smooth cycles $f_j:(V_j,\partial V_j)\to(M,\partial M)$  representing $a_j\in H_*(M,\partial M;\ZZ_2)$ for every $j=p+1,\ldots, m$,  such that all the maps are in general position. Then  $GW_{A,p,m}(a_1,\ldots, a_m)$ counts  the parity of $J$-holomorphic spheres in class $A\in H_2^S(M)$, such that  $z_i$ is mapped to $f_i(V_i)$ for every  $i=1,\ldots, p$, and $z_j$ is mapped to $f_j(V_j\minus\del V_j)$ for every $j=p+1,\ldots, m$.\\
If $\FF=\CC$, first define $GW_{A,p,m}$ for rational classes. For  non-zero  $a_i\in H_*(M;\QQ)$, $i=1,\ldots, p$, there exist non-zero $r_i\in\QQ$ and  smooth cycles $f_i: V_i\to M$ representing $r_ia_i$, and for non-zero $a_j\in H_*(M,\partial M;\QQ),\ j=p+1,\ldots, m$,  there exist non-zero $r_j\in\QQ$ and relative smooth cycles $f_j:(V_j,\partial V_j)\to(M,\partial M)$  representing  $r_ja_j$, such that all the maps are in general position. Then the invariant  $GW_{A,p,m}(r_1a_1,\ldots, r_ma_m)$ counts  the algebraic number of $J$-holomorphic spheres in class $A\in H_2^S(M)$, such that  $z_i$ is mapped to $f_i(V_i)$ for every  $i=1,\ldots, p$, and $z_j$ is mapped to $f_j(V_j\minus\del V_j)$ for every $j=p+1,\ldots, m$.
Recall from \cite[Section 1.2]{Lan} that $GW_{A,p,m}$ defined on integral cycles is multilinear.  We can therefore define $$GW_{A,p,m}(a_1,\ldots, a_m):=\frac{1}{r_1\cdots r_m}GW_{A,p,m}(r_1a_1,\ldots, r_ma_m).$$ Next, we extend the rational $m$-linear map $GW_{A,p,m}$ to a complex $m$-linear map by $\CC$-linearity.

Next, consider the following Novikov ring $\Lambda$. Let
\beq\label{equation: Gamma group of spherical classes}
\Gamma:=\Gamma(M,\omega):=\frac{H_2^S(M)}{\ker( c_1)\cap \ker(\omega)}
\eeq
and let
\beq\label{equation: G group of periods}
G:=G(M,\omega):=\frac12\omega\(H_2^S(M)\)\subseteq\RR
\eeq
be the subgroup of half-periods of the symplectic form $\omega$ on spherical homology classes.  Let $s$ be a formal variable. Define the field $\KK_G$ of generalized Laurent series in  $s$ over $\FF$ of the form
\beq\label{equation:  field KG of generalized Laurent series}
\KK_G:=\left\{f(s)=\sum\limits_{\alpha\in G}z_{\alpha}s^{\alpha}, z_{\alpha}\in\FF|\ \#\{\alpha>c|z_{\alpha}\neq 0\}<\infty,\ \forall c\in\RR \right\}
\eeq

\begin{defn}\label{defn: Novikov ring Lambda}
Let $q$ be a formal variable. The \textbf{Novikov ring $\Lambda:=\Lambda_G$} is the ring of polynomials in $q,q^{-1}$ with coefficients in the field $\KK_G$, i.e.
\beq\label{equation:  Novikov ring}
\Lambda:=\Lambda_G:=\KK_G[q,q^{-1}].
\eeq
We equip the ring $\Lambda_G$ with the structure of a graded ring by setting $\deg(s)=0$ and $\deg(q)=1$.  We shall denote by $\Lambda_k$ the set of  elements of $\Lambda$ of degree $k$. Note that  $\Lambda_0=\KK_G$.
\end{defn}

Now, the  quantum homology  $QH_*(M;\Lambda)$ and the relative  quantum homology  $QH_*(M,\partial M;\Lambda)$ are defined as follows. As  modules, they are graded modules over $\Lambda$ defined by
$$
QH_*(M;\Lambda):=H_*(M;\FF)\otimes_{\ff}\Lambda,\;\;QH_*(M,\partial M;\Lambda):=H_*(M,\partial M;\FF)\otimes_{\ff}\Lambda.
$$
A grading on both modules is given by $\deg(a\otimes zs^{\alpha}q^m)=\deg(a)+m$.  Next, we define the quantum products $\ast_l, l=1,2,3$, which are  deformations of  the classical intersection products $\bullet_l, l=1,2,3$. Choose a homogeneous basis  $\{e_k\}_{k=1}^{d}$ of $H_*(M;\FF)$, such that $e_1=[pt]\in H_0(M;\FF)$. Let  $\{e^{\vee}_k\}_{k=1}^{d}$ be the dual homogeneous basis of $H_*(M,\partial M;\FF)$ defined by $\la e_i, e^{\vee}_j \ra=\delta_{ij}$, where $\la\cdot,\cdot\ra$ is the Kronecker pairing.

\begin{defn}\label{defn: quantum products}
Let $A\in H_2^S(M)$ and let $[A]\in\Gamma$ be the image of $A$ in $\Gamma$.
Bilinear homomorphisms of $\Lambda$-modules
\bea\label{equation: quantum intersection products}
&\ast_1:QH_*(M;\Lambda)\times QH_*(M;\Lambda)\to QH_*(M;\Lambda) \\
&\ast_2:QH_*(M;\Lambda)\times QH_*(M,\partial M;\Lambda)\to QH_*(M;\Lambda) \\
&\ast_3:QH_*(M,\partial M;\Lambda)\times QH_*(M,\partial M;\Lambda)\to QH_*(M,\partial M;\Lambda)
\eea
are given as follows.\\ Let $a\in H_i(M;\FF), b\in H_j(M;\FF)$ and let $c\in H_i(M,\partial M;\FF), d\in H_j(M,\partial M;\FF)$. Then

\bea\label{equation: formula of quantum intersection products}
& a\ast_1 b:=\sum_{[A]\in\Gamma}\(\sum_{i=1}^d\sum_{A'\in[A]}GW_{A',2,3}(a,b,e^{\vee}_i)e_i\)\otimes s^{-\omega(A)}q^{-2c_1(A)},
\\
&a\ast_2 d:=\sum_{[A]\in\Gamma}\(\sum_{i=1}^d\sum_{A'\in[A]}GW_{A',1,3}(a,d,e^{\vee}_i)e_i\)\otimes s^{-\omega(A)}q^{-2c_1(A)},
\\
&c\ast_3 d:=\sum_{[A]\in\Gamma}\(\sum_{i=1}^d\sum_{A'\in[A]}GW_{A',1,3}(c,d,e_i)e^{\vee}_i\)\otimes s^{-\omega(A)}q^{-2c_1(A)},
\eea
with $\deg(a\ast_1 b)=\deg(a\ast_2 d)=\deg(c\ast_3 d)=i+j-2n$. We extend these $\FF$-bilinear homomorphisms on classical homologies to $\Lambda$-bilinear homomorphisms on quantum homologies by $\Lambda$-linearity.
\end{defn}

The next theorem summarizes the main properties of  quantum products. The proof is standard -- see \cite[Theorem 1.3.4]{Lan} for more details.
\begin{thm}\label{thm:properties of quantum products}
\mbox{}
\be
\item[$(i)$] The quantum products $\ast_l, l=1,2,3$ are  super-commutative in the sense that
$
(a\otimes 1)\ast_l(b\otimes 1)=(-1)^{\deg(a)\deg(b)}(b\otimes 1)\ast_l(a\otimes 1)
$
for $ l=1,2,3$ and for elements $a,b\in H_*(M;\FF)\cup H_*(M,\partial M;\FF)$ of pure degree.
\item[$(ii)$]The triple $(QH_*(M;\Lambda), +, \ast_1)$ has the structure of a non-unital associative $\Lambda$-algebra.
\item[$(iii)$]The triple $(QH_*(M,\partial M;\Lambda), +, \ast_3)$ has the structure of a unital associative $\Lambda$-algebra, where $[M,\partial M]$ is the multiplicative unit.
\item[$(iv)$] The quantum product $\ast_2$ defines on $(QH_*(M;\Lambda), +, \ast_1)$ the structure of an associative algebra over the algebra $(QH_*(M,\partial M;\Lambda), +, \ast_3)$.
\ee
\end{thm}

Like in the closed case, we have different natural pairings. The $\KK_G$-valued pairings are given by
\bea\label{equation: Delta pairing on QH}
&\Delta_1:QH_k(M;\Lambda)\times QH_{2n-k}(M;\Lambda)\to\Lambda_0=\KK_G,\\
&\Delta_2:QH_k(M;\Lambda)\times QH_{2n-k}(M,\partial M;\Lambda)\to\Lambda_0=\KK_G,\\
&\Delta_l\(a, b\):=\imath(a\ast_l b),\ \text {for}\ l=1,2,
\eea
where $\imath:QH_0(M;\Lambda)=\bigoplus_iH_i(M;\FF)\otimes_{\ff}\Lambda_{-i}\to\KK_G$ is the map that sends $[pt]\otimes f_0(s)+\sum_{m=1}^{2n} a_m\otimes f_m(s)q^{-m}$ to $f_0(s)$.  The $\FF$-valued pairings are given by
\beq\label{equation: Pi pairing on QH}
\Pi_l=\jmath\circ\Delta_l,\ \text {for}\ l=1,2,
\eeq
where the map $\jmath:\KK_G\to\FF$ sends $f(s)=\sum_{\alpha}z_{\alpha}s^{\alpha}\in \KK_G$ to $z_0$. The following proposition is proved similarly to its analogue in the closed case -- see \cite[Proposition 1.3.6]{Lan} for details.
\begin{prop}\label{prop: properties of Delta-Pi pairings}
The pairings $\Delta_l$ and $\Pi_l$, $l=1,2$,  satisfy
\bea\label{equation: Delta-Pi}
&\Delta_l(a,b)=\Delta_2(a\ast_l b,[M,\partial M]\otimes 1),\\
&\Pi_l(a,b)=\Pi_2(a\ast_l b,[M,\partial M]\otimes 1),
\eea
for any quantum homology classes $a\in QH_k(M;\Lambda), b\in QH_{2n-k}(M,\partial M;\Lambda)$. Moreover, the pairings $\Delta_2$ and $\Pi_2$ are non-degenerate.
\end{prop}

The ring $\Lambda$ and the groups $QH_*(M,\partial M;\Lambda)$, $QH_*(M;\Lambda)$  admit a valuation by extending the valuation $\nu:\KK_G\to G\cup\{-\infty\}$ on the field $\KK_G$, which is  given by
$$
\begin{cases}
&\nu\(f(s)=\sum\limits_{\alpha\in G}z_{\alpha}s^{\alpha}\):=\max\{\alpha| z_{\alpha}\neq 0\},\; f(s)\not\equiv 0\\
&\nu(0)=-\infty.
\end{cases}
$$
Extend $\nu $ to $\Lambda$ by $\nu(\lambda):=\max\{\alpha|p_{\alpha}\neq 0\}$, where $\lambda$ is uniquely represented by $\lambda=\sum\limits_{\alpha\in G}p_{\alpha}s^{\alpha},\;\; p_{\alpha}\in\FF[q,q^{-1}]$.
Note that for all $\lambda, \mu\in\Lambda$ we have
$$
\nu(\lambda+\mu)\leq\max(\nu(\lambda), \nu(\mu)),\;\nu(\lambda\mu)=\nu(\lambda)+\nu(\mu),\;\nu(\lambda^{-1})=-\nu(\lambda).
$$
Now,  any non-zero $a\in QH_*(M,\partial M;\Lambda)$ (resp.
$a\in QH_*(M;\Lambda)$ ) can be uniquely written as $a=\sum_ia_i\otimes\lambda_i$, where $a_i\in H_*(M,\partial M; \FF)$ (resp. $a_i\in H_*(M; \FF)$ ) and $\lambda_i\in\Lambda$. Define $\nu(a):=\max_i\{\nu(\lambda_i)\}$.

Finally, let us note that there is a natural extension of the theory of Hamiltonian fibrations over the two-sphere $\SSS^2$ with a closed fiber to the case of compact convex symplectic fiber $(M,\omega)$. Namely, we consider fibrations $M\hookrightarrow P\overset{\pi}{\twoheadrightarrow}\SSS^2,$ with the fiber $(M,\omega)$ over $\SSS^2$ with the structural group $\Ham_c(M,\omega)$ that are trivial near the boundary, i.e. there exists an open neighborhood $W$ of  $\del P$ such that $\left. \pi\right|_W$ can be identified with the projection $\SSS^2 \times U\to \SSS^2$, where $U$  is an open neighborhood of $\del M$. Recall (see e.g. \cite[Section 8.2]{MS3}) that a Hamiltonian fibration always comes with a certain (Hamiltonian) connection.

Now, consider the space $\fr{J}(P,\partial P,\omega)$ of almost complex structures on $P$ that are \emph{\textbf{adapted to the fibration}} (cf. \cite{Seidel}, \cite[Section 8.2]{MS3}). Equip $\SSS^2$ with a positive oriented complex structure $j_{\sss^2}$. We say
$\JJ\in\fr{J}(P,\partial P,\omega)$ iff
\be
\item[(i)]the restriction of $\JJ$ to each fiber $M_z$ belongs to the space $\CL{J}(M_z,\partial M_z,\omega_z)$ of $\omega_z$-compatible  almost complex structures that are adapted to the boundary,
\item[(ii)] on $W=\SSS^2 \times U$ the structure $\JJ$ does not depend on the $z$-variable, for $z\in\SSS^2$, and it is of the form $\JJ|_{\{z\}\times U}\oplus j_{\sss^2}$,
\item[(iii)]the horizontal distribution $\mathrm{Hor}\subset TP$ w.r.t the Hamiltonian connection on the fibration is invariant under $\JJ$,
\item[(iv)]the projection $\pi:P \twoheadrightarrow\SSS^2$ is a $(\JJ,j_{\sss^2})$-holomorphic map.
\ee
Like in the closed case, one can define a sectional Gromov-Witten invariant relative to the boundary $SGW_{A,p,m}$ as an $m$-linear map over $\FF$
$$
SGW_{A,p,m}: H_*(M;\FF)^{\times p}\times H_*(M,\partial M;\FF)^{\times (m-p)}\to\FF,
$$
which counts  the following geometric configurations. Fix distinct marked points $(z_1,\ldots,z_m)\in\(\SSS^2\)^m$ and let $\iota_i:M\hookrightarrow P$ be a symplectic embedding with $\iota_i(M)=\pi^{-1}(z_i),\, i=1,...,m$.\\
If $\FF=\ZZ_2$, choose  smooth cycles $f_i: V_i\to M$  representing $a_i\in H_*(M;\ZZ_2)$ for every $i=1,\ldots, p$, and choose relative smooth cycles $f_j:(V_j,\partial V_j)\to(M,\partial M)$ representing $a_j\in H_*(M,\partial M;\ZZ_2)$ for every $j=p+1,\ldots, m$,  such that all the maps are in general position. Then  $SGW_{A,p,m}(a_1,\ldots, a_m)$ counts  the parity of $(j_{\sss^2},\JJ)$-holomorphic sections in class $A\in H_2^S(P)$, such that  $z_i$ is mapped to $\iota_i\circ f_i(V_i)$ for every  $i=1,\ldots, p$, and $z_j$ is mapped to $\iota_j\circ f_j(V_j\minus\partial V_j)$ for every $j=p+1,\ldots, m$.\\
If $\FF=\CC$, then, as before, we first define $SGW_{A,p,m}$ for rational classes by counting the algebraic number of $(j_{\sss^2},\JJ)$-holomorphic sections in class $A\in H_2^S(P)$ that map the marking points to the corresponding rational cycles, and then extend it to a complex $m$-linear map by $\CC$-linearity -- see \cite[Sections 1.1.10, 1.2.6]{Lan} for more details.

\subsection{Floer homology of compact convex symplectic manifolds.}
We extend the definitions and the constructions of \cite{F-S} to the case of strongly semi-positive compact convex symplectic manifolds. The proofs are easy modifications of the proofs from \cite{F-S} -- see \cite[Section 2.1]{Lan} for more details.

\subsubsection*{Admissible Hamiltonians on $M$.}\label{subsection: admissible Hamiltonians on M}
First, recall that  a smooth function (Hamiltonian) $H:\SSS^1\times M\to\RR$ generates a smooth vector field $X_H:\SSS^1\times M\to TM$, called the Hamiltonian vector field of $H$, by $\omega(X_H,\cdot)=-dH(\cdot)$. The flow generated by $X_H$ will be denoted by $\phi^t_H$.

Let $X$ be  a Liouville  vector field on a neighborhood of $\partial M$ -- see Definition~\ref{Def:Convex symplectic manifolds}. Using $X$ we can symplectically identify a neighborhood of $\partial M$ with $\( \partial M \times (-2\ve, 0], d \( e^r \alpha \) \)$ for some $\ve >0$, where $\alpha=\iota_X\omega$ is the Liouville $1$-form. In this identification we used coordinates $(x,r)$ on $\partial M \times (-2\ve, 0]$, and in these coordinates, $X(x,r) = \frac{\partial}{\partial r}$
on $\partial M \times (-2\ve, 0]$. We can thus view $M$ as a compact subset of the non-compact symplectic manifold $(\hat{M},\hat{\omega})$
defined as
\begin{eqnarray*}
\hat{M}      &=& M \cup_{\partial M \times \{0\}} \partial M \times [0,\infty), \\
\hat{\omega} &=&
   \left\{ \begin{array}{lll}
          \omega & \text{on} & M, \\
          d \( e^r \alpha \) & \text{on} & \partial M \times (-2\ve,\infty),
           \end{array}
   \right.
\end{eqnarray*}
and $X$ smoothly extends to $\partial M \times (-2\ve, \infty)$ by
$$
\hat{X} (x,r) := \frac{\partial}{\partial r}, \quad\,
(x,r) \in \partial M \times (-2\ve, \infty) .
$$
For any $r\in\RR$ we denote the open ``tube"
$\partial M \times (r, \infty)$ by $P_r$. Let $\phi^t:=\phi^t_{\hat{X}}$ be the flow of $\hat{X}$.
Then $\phi^r(x,0) = (x,r)$ for $(x,r) \in P_{-2\ve}$.
Choose an $\hat{\omega}$-compatible almost complex structure $\hat{J}$ on $\hat{M}$, such that
\begin{eqnarray}
&&\hat{\omega} \( \hat{X}(x),\hat{J} (x) v \) = 0,  \hspace{25mm}  x \in \partial M, \,\, v \in T_x \partial M,  \label{J3} \\
&&\hat{\omega} \( \hat{X}(x), \hat{J}(x) \hat{X}(x) \) = 1,  \hspace{18mm}  x \in \partial M, \label{J2} \\
&&d_{(x,0)} \phi^r  \hat{J}(x,0) = \hat{J} (x,r) d_{(x,0)} \phi^r,   \hspace{7mm}  (x,r) \in P_{-2\ve}, \label{J1}
\end{eqnarray}
\begin{defn}\label{defn: parametrized admissible a.c.s}
For any smooth  manifold $B$ define the subset $\hCL{J}_B$ of the set of smooth sections $\Gamma \big( \hat{M} \times B, \mathrm{End} \big(T \hat{M} \big) \big)$ by
$$
\hat{J} \in \hCL{J}_B\ \IFF\  \hat{J}_b := \hat{J} (\cdot,b)\ \text{is}\ \hat{\omega}\text{-compatible and satisfies \eqref{J3}, \eqref{J2} and \eqref{J1}.}
$$
For any $r\geq-2\ve$ define  $\hCL{J}_{B,P_r}$ to be the set of all $\hat{J}\in\hCL{J}_B$ that are independent of the $b$-variable on $\cl{P_r}$. And, at last,  we define the set
\beqn
\CL{J}_{B, P_r} := \left\{ J \in \Gamma \(M \times B ,\mathrm{End}(TM) \) \mid J=\hat{J}|_{M \times B} \text{ for some } \hat{J} \in \hCL{J}_{B, P_r}\right\}.
\eeq
\end{defn}
By \cite[ Remark $4.1.2$]{BPS} or \cite[ discussion on page $106$]{CFH}, the space $\CL{J}_{B, P_r}$ is non-empty and  connected.

Let $R$ be the Reeb vector field of the Liouville form $\iota_X\omega$ on $\partial M$, i.e.  $R = \hat{J} \,\hat{X} |_{\partial M}$ for any $\hat{J}\in\hCL{J}_{B, P_{-2\ve}}$.  Let  $e \in  C^\infty \( P_{-2\ve} \)$ be given by $e(x,r) := e^r$. We have that for any $h \in  C^\infty ( \RR )$,  the Hamilton equation
$\dot{x} = X_H (x)$ of $H = h \circ e$ restricted to $\partial M$ has the form
\begin{equation}  \label{hamilton}
\dot{x}(t) \,=\, h'(1) \,R(x(t)) .
\end{equation}
Define the Reeb period $\kappa \in (0, \infty ]$ of $R$ by
\bea\label{equation: Reeb period kappa}
\kappa :=\inf_{c>0} \left\{\dot{x}(t) = c\, R(x(t)) \text{ has a non-constant  $1$-periodic solution} \right\} .
\eea
Define  two sets of smooth functions $\hCL{H}^{\pm}\subseteq C^\infty(\SSS^1 \times \hat{M})$ by
\bea\label{equation: preadmissible hamiltonians}
&\hat{H}\in\hCL{H}^+\Leftrightarrow \exists h\in C^\infty( \RR )\ \text{such that}\
\begin{cases}
&0 \leq h'(e^r) < \kappa\ \forall r\ge 0,\\
&h'(e^r)=0\ \forall r\ge\ve,\\
&\hat{H}|_{\sss^1 \times \cl{P_0} }= h \circ e,
\end{cases}
\\
&\hat{H}\in\hCL{H}^-\Leftrightarrow-\hat{H}\in\hCL{H}^+.
\eea
By \eqref{hamilton}, \eqref{equation: Reeb period kappa} and \eqref{equation: preadmissible hamiltonians}\,, we have that for any $\hat{H} \in \hCL{H}^{\pm}$ the restriction of the flow $\phi_{\hat{H}}^t$  to $\cl{P_0}$  has no non-constant $1$-periodic solutions. Next, we define two sets of \emph{\textbf{admissible Hamiltonian functions on $M$}} by
$$
\CL{H}^{\pm} := \left\{ H \in  C^\infty (\SSS^1 \times M ) \mid  H = \hat{H}|_{\sss^1 \times M} \text{ for some } \hat{H} \in \hCL{H}^{\pm} \right\}.
$$
Note that  the space  $\CL{H}_c(M)$ of $C^\infty$-smooth functions on  $\SSS^1\times M$, whose support is compact and is contained in $\SSS^1\times(M\minus\partial M)$ is the subset of $\CL{H}^+\cap\CL{H}^-$.

\subsubsection*{Floer homology groups.}
Given $H\in\CL{H}^{\pm}$, we denote the set of contractible $1$-periodic orbits of $\phi_H^t$ by $\PH{H}$. For generic $H$ we have that
$$
\det\( \id - d_{x(0)} \phi^1_{H}\) \neq 0,\,\forall\,x \in \PH{H},
$$
see \cite[Theorem 3.1]{H-S}. Since $M$ is compact, $\PH{H}$ is a finite set. Such a generic admissible $H$  is called {\it \textbf{regular}}, and the set of regular admissible Hamiltonians is denoted by $\CL{H}^{\pm}_{reg} \subset \CL{H}^{\pm}$.
Let $\Lambda$ be the Novikov ring defined in \eqref{equation:  Novikov ring}.
Given $H\in \CL{H}^{\pm}$, consider the \emph{\textbf{symplectic action functional}}
$$
\CL{A}_H(\x) \,:=\, -\int_{\dd^2} \bar{x}^*\omega+\int_{\sss^1}H(t, x(t)) \,dt,
$$
which is defined on the space of equivalence classes $\x:=[x,\bar{x}]$ of pairs $(x,\bar{x})$ w.r.t the equivalence relation $(x_1,\bar{x}_1) \sim (x_2,\bar{x}_2)$ iff $ x_1 = x_2, \  \omega(\bar{x}_1 \#(- \bar{x}_2)) =c_1(\bar{x}_1 \#(- \bar{x}_2)) =0,$ where $x: \SSS^1\to M$ is a smooth contractible loop and $\bar{x}: \DD^2 \to M$ satisfies $\bar{x}(e^{it})=x(t)$.
For a generic pair  $(H,J)\in\CL{H}^{\pm}_{reg}\times\CL{J}_{\sss^1, P_{-2\ve}}$ one can associate to $\CL{A}_H$ a Floer chain complex $(CF_*(H;\Lambda), \partial_{H,J})$: the generators of the complex are the elements of $\tPH{H}=\{\x=[x,\bar{x}]\,|\, x\in\PH{H} \}$ and
 the $\Lambda$-linear  differential $\partial_{H,J}:CF_*(H;\Lambda)\to CF_{*-1}(H;\Lambda)$ counts the algebraic number of isolated $(H,J)$-Floer cylinders connecting  elements of  $\PH{H}$, such that the corresponding capped cylinders represent the zero class in $\Gamma$. By the maximum principle\,, see \cite[Corollary 2.3]{F-S}, $(H,J)$-Floer cylinders lie away from the boundary $\del M$ and thus, the differential $\partial_{H,J}$ is well-defined. Such a generic pair $(H,J)$ will be called a \emph{\textbf{regular admissible pair}}.
The homology  $H_*(CF_*(H;\Lambda), \partial_{H,J})$ is called the \emph{\textbf{Floer homology over $\Lambda$}} and will be denoted by $HF_*(H,J;\Lambda)$. By  \cite[Corollary 2.3]{F-S} and by the connectedness of the space $\CL{J}_{\sss^1, P_{-2\ve}}$, the Floer continuation maps theorem \cite[ Theorem $5.2$]{H-S} still holds. Thus, the Floer homology $HF_*(H,J;\Lambda)$ does not depend on the choice of a regular admissible pair $(H,J)$.

We have two well-defined pair-of-pants products according to the class of Hamiltonians. Namely, for generic pairs $(H^\pm_i,J^\pm_i)\in\CL{H}^{\pm}_{reg}\times\CL{J}_{\sss^1, P_{-2\ve}},\,i=1, 2, 3$, we have two homomorphisms
$$
\ast_{PP}^\pm:HF_k(H^\pm_1,J^\pm_1;\Lambda)\otimes HF_l(H^\pm_2,J^\pm_2;\Lambda)\to HF_{k+l-2n}(H^\pm_3,J^\pm_3;\Lambda).
$$
Following the arguments of U. Frauenfelder and F. Schlenk in \cite[ Section $4$]{F-S} we have the Piunikhin-Salamon-Schwarz-type isomorphisms of $\Lambda$-algebras
$$
\Phi_{PSS}^\pm : \(QH^\pm_*,\ast^\pm\)\to\(HF_*\(H^\pm,J^\pm;\Lambda\), \ast_{PP}^\pm\),
$$
where
$\(QH^+_*, \ast^+\):= \(QH_*(M,\partial M;\Lambda), \ast_3\)$ and $\(QH^-_*, \ast^-\):= \(QH_*(M;\Lambda), \ast_1\)$ -- see \cite[Section 2.1.6]{Lan} for more details.

\subsection{Spectral invariants.}\label{subsection: spectral invariants}
The symplectic action functional $\CL{A}_H$ and the valuation $\nu$  define a filtration $CF^{(-\infty,\alpha)}_*(H,\Lambda)$ on the Floer chains group in a standard way --  see e.g. \cite{Oh1}. The differential $\partial_{H,J}$ preserves the subspace $CF^{(-\infty,\alpha)}_*(H,\Lambda)$ and thus, the filtration descends to the homology:
$$
HF^{(-\infty,\alpha)}_*(H,J;\Lambda) := H_*(CF^{(-\infty,\alpha)}_*(H,\Lambda),\partial_{H,J}).
$$

Following Viterbo \cite{V1}, Schwarz \cite{Sch} and Oh \cite{Oh1} we define two kinds of spectral invariants
$
c^\pm: QH^\pm_*\times\CL{H}^\pm\to\RR,
$
which descend to
$
\tHam_c(M,\omega).
$
Namely, for a regular admissible pair $(H^\pm, J^\pm)$  and for $0\neq a^\pm\in QH^\pm_*$ we define
\beq\label{equation:selectors in H^pm_reg}
c^\pm(a^\pm, H^\pm):=\inf\left\{\alpha\in\RR\ | \ \Phi_{PSS}^\pm(a^\pm)\in HF_*^{(-\infty,\alpha)}(H^\pm, J^\pm;\Lambda)\right\}.
\eeq

Following   \cite{Sch},  \cite[Theorem $5.3$, Proposition $5.8$]{Oh Spec for general1}, \cite{Usher2}, we obtain the following standard properties of the spectral numbers -- see \cite[Section 2.2] {Lan} for details.
\begin{prop}\label{prop: well definition of c^pm}
\mbox{}

\item[$(i)$]Spectral numbers $c^\pm(a^\pm, H^\pm)$ are finite and do not depend on the choice of almost complex structures $J^\pm$ .
\item[$(ii)$]For any $H^\pm,K^\pm\in\CL{H}_{reg}^\pm$, we have
  $$
  \big|c^\pm(a^\pm, H^\pm)-c^\pm(a^\pm, K^\pm)\big|\leq\hn{H^\pm-K^\pm},
  $$
where the \emph{\textbf{$L^{(1,\infty)}$-norm}} $\hn{\cdot}$ on $ C^{\infty}([0,1]\times M)$ is defined as
 $$\hn{H}:=\int\limits^1_0\left(\sup\limits_{x\in M}H(t,x)-\inf\limits_{x\in M}H(t,x)\right)dt.$$
In particular, the functions $H^\pm\mapsto c^\pm(a^\pm, H^\pm)$ are $C^0$-continuous on $\CL{H}^\pm_{reg}$.
\end{prop}

\begin{cor}\label{cor:extension of selectors to compact support}
The functions $c^\pm:H^\pm\mapsto c^\pm(a^\pm, H^\pm)$ can be extended to  functions $c^\pm:\CL{H}^\pm\to\RR$, which are 1-Lipschitz w.r.t. the $L^{(1,\infty)}$-norm. In particular, $c^\pm(a^\pm, H)$ are well defined for $C^\infty$-smooth compactly supported $H\in\CL{H}_c(M)\subset\CL{H}^+\cap\CL{H}^-$.
\end{cor}

\begin{prop}\label{prop: properties of c_pm on H_reg}
For any $H^\pm, K^\pm\in\CL{H}^\pm$ and any $0\neq a^\pm, b^\pm\in QH^\pm_*$ we have the following properties of spectral numbers.

\item{{\bf (Spectrality)}}\ If $(M,\omega)$ is rational, i.e. the group $\displaystyle \omega\(H_2^S(M)\)\subset\RR$ is a discrete subgroup of $\RR$, or if $(M,\omega)$ is irrational, but the Hamiltonians $H^\pm$ are non-degenerate, then $c^\pm(a^\pm, H^\pm)\in \Spec (H^\pm)$, where $\Spec(H^\pm)$ is the action spectrum, i.e. the set of critical values of  $\CL{A}_{H^\pm}$.

\item{{\bf (Quantum homology shift property)}}\ $c^\pm (\lambda a^\pm, H^\pm)
= c^\pm (a^\pm, H^\pm) + \nu (\lambda)$ for all $\lambda \in \Lambda$.

\item{{\bf (Monotonicity)}}\ If $H^\pm\leq K^\pm$, then $c^\pm
(a^\pm, H^\pm)\leq c^\pm (a^\pm, K^\pm)$.

\item{{\bf ($ C^0$-continuity)}}\ $|c^\pm(a^\pm, H^\pm)-c^\pm(a^\pm, K^\pm)|\leq\hn{H^\pm-K^\pm}.$

\item{{\bf (Symplectic invariance)}}\ $c^\pm(a^\pm,\psi^*H^\pm) = c^\pm (a^\pm,H^\pm)$
for every element $\psi \in \Symp^0_c (M,\omega)$, where $\Symp^0_c(M,\omega)$ denotes the identity component of the group of symplectomorphisms of  $(M,\omega)$,  whose support is contained in $M\minus\partial M$.

\item{{\bf (Normalization)}}\ $c^\pm (a^\pm,0) = \nu (a^\pm)$ for every $a^\pm\in QH^\pm_*$.

\item{{\bf (Homotopy invariance)}}\ If $H, K\in\CL{H}_c(M)$ and $\til{\phi^1_H}=\til{\phi^1_K}$, then  $c^\pm (a^\pm, H) = c^\pm (a^\pm, K)$.  Thus one can define
$c^\pm (a^\pm,\til{\phi})$ for any $\til{\phi}\in\tHam_c (M, \omega)$ as $c^\pm (a^\pm,H)$ for any  $H\in\CL{H}_c(M)$
generating $\til{\phi}$, i.e. $\til{\phi}=\til{\phi^1_H}$.

\item{{\bf (Triangle inequality)}}\ For any  $H,K\in\CL{H}_c(M)$ then
$$c^\pm (a^\pm\ast^\pm b^\pm,H\#K)\leq c^\pm (a^\pm, H) + c^\pm (b^\pm, K),$$ and thus, for any $\til{\phi}, \til{\psi}\in\tHam_c (M, \omega)$ we have
$$c^\pm (a^\pm\ast^\pm b^\pm,\til{\phi}\til{\psi} )\leq c^\pm (a^\pm, \til{\phi}) + c^\pm (b^\pm, \til{\psi}).$$

\item{{\bf (Poincar\'e-Lefschetz duality)}}\ Let $\Pi_2: QH_*^+\times QH^-_{2n-*}\to\FF$
be the non-degenerate pairing defined in ~\eqref{equation: Pi pairing on QH}. Then
$$
 c^\pm (a^\pm, H^\pm) =-\inf\left\{ c^\mp\(b^\mp, \(H^\pm\)^{(-1)}\)\ \vline\ \Pi_2(a^\pm, b^\mp)\neq 0\right\},
$$
where the inverse Hamiltonian $H^{(-1)}$ generates the inverse path $(\phi^t_H)^{-1}$, and  is given by $H^{(-1)}(t,p):=-H(-t,p)$. In particular, for any $\til{\phi}\in\tHam_c (M, \omega)$ we have
$
 c^\pm\(a^\pm, \til{\phi}\) =-\inf\left\{ c^\mp \(b^\mp,\(\til{\phi}\)^{-1}\)\ \vline\ \Pi_2(a^\pm, b^\mp)\neq 0\right\}.
$

\end{prop}
\bigskip

\section{Quasi-morphisms and  quasi-states.}  In this section we extend the constructions from \cite{E-P1} and \cite{E-P2} of quasi-morphisms and  quasi-states  to the case of compact convex strongly semi-positive symplectic manifolds.

\subsection{Quasi-morphisms}
Fix two non-zero idempotents $\e^\pm\in QH^\pm_{2n}$.

\begin{thm}\label{thm: quasi-morphism q on hamtil} Suppose that
\item{$\mathbf{(QM_1)}$} the $\KK_G$-subalgebras $\Bbbk^\pm:=\e^\pm QH^\pm_{2n}$ are fields,
\item{$\mathbf{(QM_2)}$} for all $a^\mp\in QH_0^\mp$ one has $\Pi_2(\e^\pm, a^\mp)\neq 0 \then \e^\mp\ast^\mp a^\mp\neq 0$.\\
Then there exist constants $C^\pm>0$, such that
$$c^\mp\(\e^\mp, \til{\phi}\)+c^\pm\(\e^\pm, \til{\phi}^{-1}\)\leq C^\mp$$
for all $\til{\phi}\in\tHam_c(M, \omega)$. It implies that for all $\til{\phi}\in\tHam_c(M, \omega)$
$$
c^+\(\e^+,\til{\phi}\)+c^+\(\e^+,\til{\phi}^{-1}\)+
c^-\(\e^-,\til{\phi}\)+c^-\(\e^-,\til{\phi}^{-1}\)\leq C:=C^++C^-.
$$
In particular, since
$$0\leq\nu\(\e^\pm\)=c^\pm\(\e^\pm,\til{\id_M}\)= c^\pm\(\e^\pm\ast^\pm\e^\pm, \til{\phi} \til{\phi}^{-1}\) \leq c^\pm\(\e^\pm, \til{\phi}\)+c^\pm\(\e^\pm, \til{\phi}^{-1}\)$$
we have that
\beq\label{equation: positivity of q^pm}
\max\left\{c^+\(\e^+,\til{\phi}\)+c^+\(\e^+,\til{\phi}^{-1}\); c^-\(\e^-,\til{\phi}\)+c^-\(\e^-,\til{\phi}^{-1}\)\right\}\leq C
\eeq
for any $\til{\phi}\in\tHam_c(M, \omega)$.
\end{thm}
\begin{rem}
Note the behaviour of the $\pm$ sign vs the $\mp$ sign in the $\mathbf{(QM_2)}$ condition.
We also note that there is no such $\mathbf{(QM_2)}$ condition in the closed case. Moreover, we do not know any example of a non-closed manifold, where $\mathbf{(QM_2)}$ does not hold.
\end{rem}
Recall that the Hofer (bi-invariant) pseudo-metric $\til{d}$ on $\tHam_c(M,\omega)$ is given by
$\til{d}(\til{\phi^1_H},\til{\phi^1_K}):=\|(\til{\phi^1_H})^{-1}\til{\phi^1_K}\|$, where $\|\til{\phi}\|:=\inf\{\hn{H} \ |\ \til{\phi^1_H}=\til{\phi} \}$.

\begin{cor}\label{cor: quasi-morphism r on hamtil}
The functions $c^\pm:\tHam_c(M,\omega)\to\RR$ given by
\beq\label{equation: quasimorphism r}
c^\pm\(\til{\phi}\):=c^\pm\(\e^\pm,\til{\phi}\)
\eeq
are quasi-morphisms. These functions satisfy
\beq\label{equation: C0 continuity of qm}
\left|c^\pm\(\til{\phi^1_H}\)-c^\pm\(\til{\phi^1_K}\)\right|\leq\hn{H-K},
\eeq
and hence,  are $1$-Lipschitz w.r.t the Hofer  pseudo-metric $\til{d}$. Moreover, for any $\til{\phi}\in\tHam_c(M, \omega)$ we have $\hat{c^+}(\til{\phi})=\hat{c^-}(\til{\phi})$. Denote this homogeneous quasi-morphism by $\qm$.
\end{cor}
\begin{proof}
By the triangle inequality for spectral numbers  we have
$$
c^\pm\(\e^\pm,\til{\phi}\til{\psi}\)=c^\pm\(\e^\pm\ast^\pm\e^\pm,\til{\phi}\til{\psi}\)\leq c^\pm\(\e^\pm,\til{\phi}\)+c^\pm\(\e^\pm,\til{\psi}\)
$$
for any $\til{\phi}, \til{\psi}\in\tHam_c(M, \omega)$. Hence,  $c^\pm\(\til{\phi}\til{\psi}\)\leq c^\pm\(\til{\phi}\)+c^\pm\(\til{\psi}\)$.
For another inequality we note that for any $\til{\phi}, \til{\psi}\in\tHam_c(M, \omega)$
\begin{multline*}
c^\pm\(\til{\phi}\)+c^\pm\(\til{\psi}\)=c^\pm\(\til{\phi}\)+c^\pm\(\til{\phi}^{-1}\til{\phi}\til{\psi}\) \leq\\ \leq c^\pm\(\til{\phi}\)+c^\pm\(\til{\phi}^{-1}\)+c^\pm\(\til{\phi}\til{\psi}\)\stackrel{\eqref{equation: positivity of q^pm}}{\leq} C+ c^\pm\(\til{\phi}\til{\psi}\).
\end{multline*}
Inequality \eqref{equation: C0 continuity of qm} follows directly from the $C^0$-continuity property of spectral numbers.

The equality $\hat{c^+}(\til{\phi})=\hat{c^-}(\til{\phi})$ follows immediately from the homogenization of the inequality
$c^\mp\(\e^\mp, \til{\phi}\)+c^\pm\(\e^\pm, \til{\phi}^{-1}\)\leq C^\mp$.
\end{proof}
Recall that a subset $A$ of a symplectic manifold $(M,\omega)$ is called  \emph{\textbf{ displaceable}} if there exists $\phi\in \Ham_c(M,\omega)$, such that $\phi(A)\cap\cl{A}=\varnothing.$

\begin{cor}\label{cor: homogeneous quasi-morphism tilr on hamtil}
The function $\qm:\tHam_c(M,\omega)\to\RR$  is a $1$-Lipschitz (w.r.t the Hofer pseudo-metric) homogeneous quasi-morphism. Moreover, for every open and displaceable $U\subset M$, the restriction of $\qm$ to $\tHam_c(U)$ is identically zero, where  $\tHam_c(U)\subset\tHam_c(M,\omega)$ is the subgroup of elements $\til{\phi^1_H}$ generated by Hamiltonians  $H\in \CL{H}_{c}(M)$ with $\mathrm{supp}(H(t,\cdot))\subset U$ for all $t\in\SSS^1$.
\end{cor}

\begin{proof}
The proof of the Lipschitz property of $\qm$ repeats verbatim the proof of  \cite[Proposition $3.5$]{E-P1}. The last assertion follows  from the ``shift of the spectrum" trick of  Ostrover \cite{Ostr-shift trick} -- see \cite[Corollary 3.1.5]{Lan}.
\end{proof}

\begin{rem}\label{remark: possible triviality of qms}
A priori, the homogeneous quasi-morphism $\qm$ may be a homomorphism. We discuss the examples of symplectic manifolds, for which $\qm$  is not a homomorphism in Section \ref{section: non-triviality of qms}.
\end{rem}

\begin{proof}[\textbf{Proof of Theorem~\ref{thm: quasi-morphism q on hamtil}}]
We follow the argument from \cite[Theorem $3.1$]{E-P1}.\\

\item{\textbf{Step} $\mathbf{1.}$} By \cite[Lemma $3.2$]{E-P1}, there exists a positive real number $R$, such that $\nu(a)+\nu(a^{-1})\leq R$ for every $a\in\Bbbk^+\minus\{0\}\cup\Bbbk^-\minus\{0\}$.
\item{\textbf{Step} $\mathbf{2.}$} By the triangle inequality and the normalization properties of spectral numbers,  we have
\beq\label{equation: spectral numbers estimate}
c^\pm\(\e^\pm\ast^\pm a^\pm,\til{\phi}\)\leq c^\pm\(a^\pm,\til{\phi}\)+\nu\(\e^\pm\),
\eeq
for all $a^\pm\in QH_*^\pm$ and for all $\til{\phi}\in\tHam_c(M,\omega)$.
\item{\textbf{Step} $\mathbf{3.}$} Define the sets
$\KL{Q}^\mp:=\left\{a^\mp\in QH_0^\mp\ \vline\ \Pi_2\(\e^\pm, a^\mp\)\neq 0  \right\}.$
By the Poincar\'e-Lefschetz duality property of spectral numbers, we have that for every $\til{\phi}\in\tHam_c(M,\omega)$

$$
\begin{aligned}
&-c^\pm\(\e^\pm, \til{\phi}^{-1}\)= \inf_{\Kl{Q}^\mp}\left\{c^\mp\(a^\mp, \til{\phi}\)\right\}\stackrel{\eqref{equation: spectral numbers estimate}}{\geq}
-\nu(\e^\mp)+ \inf_{\Kl{Q}^\mp}\left\{c^\mp\(\e^\mp\ast^\mp a^\mp, \til{\phi}\) \right\} \stackrel{\text{{\bf(QHSP)}}}{=}\\
&=-\nu(\e^\mp)+ \inf_{\Kl{Q}^\mp}\left\{c^\mp\(q^{2n}\(\e^\mp\ast^\mp a^\mp\), \til{\phi}\)\right\}
\stackrel{\mathbf{(QM_1)}\wedge\mathbf{(QM_2)}\, \then\, \exists\, \(q^{2n}\(\e^\mp\ast^\mp a^\mp\)\)^{-1}\in\ \Bbbk^\mp\minus\{0\}}{\geq}\\
&\geq-\nu(\e^\mp)+ \inf_{\Kl{Q}^\mp}\left\{c^\mp\(\e^\mp, \til{\phi}\) - \nu\(\(q^{2n}\(\e^\mp\ast^\mp a^\mp\)\)^{-1}\)\right\}\stackrel{\text{Step $1.$+{\bf(QHSP)}}}{\geq}\\
&\geq -\nu(\e^\mp)+ c^\mp\(\e^\mp, \til{\phi}\) - R + \inf_{\Kl{Q}^\mp}\left\{\nu\(\e^\mp\ast^\mp a^\mp\)\right\}
\stackrel{\Pi_2\(\e^\pm, a^\mp\)\neq 0\ \then\ \nu\(\e^\mp\ast^\mp a^\mp\)\geq 0}{\geq}\\
&\geq -\nu(\e^\mp)+ c^\mp\(\e^\mp, \til{\phi}\) - R.
\end{aligned}
$$
Here, {\bf(QHSP)} stands for {\bf (Quantum homology shift property)} from Proposition~\ref{prop: properties of c_pm on H_reg}. It follows that
$c^\mp\(\e^\mp, \til{\phi}\)+c^\pm\(\e^\pm, \til{\phi}^{-1}\)\leq C^\mp:=\nu(\e^\mp)+R$, for every $\til{\phi}\in\tHam_c(M,\omega)$
\end{proof}
In the absence of  conditions $\mathbf{(QM_1)}$ and $\mathbf{(QM_2)}$ of Theorem~\ref{thm: quasi-morphism q on hamtil} the functions $c^\pm$ and $\hat{c^\pm}$ most likely are not  quasi-morphisms. On the other hand, they have weaker properties analogous to \cite{E-P2}.

The proof of the next theorem follows verbatim from
\cite[Theorem $7.1$]{E-P2}.

\begin{thm}\label{thm: partial quasimorphism r}
The homogenization $\hat{c^\pm}:\tHam_c(M,\omega)\to\RR$ of $c^\pm$ has the following properties:

\item{{\bf(Controlled ~quasi-additivity)}}\ Given a displaceable open set $U\subset M$, there exists a positive constant $R$, depending only on $U$, so that
    $$
    \left|\hat{c^\pm}\(\til{\phi}\til{\psi}\)-
    \hat{c^\pm}\(\til{\phi}\)-\hat{c^\pm}
    \(\til{\psi}\)\right|\leq R\cdot\min\left\{\left\|\til{\phi}\right\|_U, \left\|\til{\psi}\right\|_U\right\}
    $$
    for any $\til{\phi}, \til{\psi}\in\tHam_c(M,\omega)$.

\item{{\bf(Semi-homogeneity)}}\ $\hat{c^\pm}\(\til{\phi}^m\)=m\hat{c^\pm}\(\til{\phi}\)$ for any $\til{\phi}\in\tHam_c(M,\omega)$ and any $m\in\ZZ_{\geq 0}$.

Thus, $\hat{c^\pm}$ is a partial quasi-morphism. In addition,  the function is $1$-Lipschitz w.r.t the Hofer pseudo-metric, and for every open and  displaceable $U\subset M$, the restriction of $\hat{c^\pm}$ to $\tHam_c(U)$ is identically zero.
\end{thm}

\subsection{Symplectic quasi-states}\label{susection: quasi-states}
Denote by $\CL{H}_{cc}(M)$  the set of $ C^{\infty}$-smooth functions $H:\SSS^1\times M\to\RR$, such that  $H\in \CL{H}_{cc}(M)$ if and only if there exists  $\fr{n}_H\in C^{\infty}(\SSS^1,\RR)$, for which $H-\fr{n}_H\in\CL{H}_c(M)$. Notice, that $\CL{H}_{c}(M)\subset \CL{H}_{cc}(M)$ and $\Ham_c(M,\omega):=\{\phi^1_H|\; H\in \CL{H}_{cc}(M)\}$. For any $H\in\CL{H}_{cc}(M)$ the spectral numbers $c^\pm (a^\pm, H)$ are defined as above, and $c^\pm (a^\pm, H)=c^\pm (a^\pm, F)+\int_0^1\fr{n}_H(t)dt$, where $H=F+\fr{n}_H$ is a unique representation with $F\in\CL{H}_c(M)$ and $\fr{n}_H\in C^{\infty}(\SSS^1,\RR)$.

\begin{thm}\label{thm: quasi-state for QM12 manflds}
Suppose that the quantum homology algebras $QH^\pm_*$ of $(M, \omega)$ satisfy conditions $\mathbf{(QM_1)}$ and $\mathbf{(QM_2)}$
(see Theorem~\ref{thm: quasi-morphism q on hamtil}). Then $ C_{cc}(M)$ admits a symplectic quasi-state (possibly trivial, see Remark \ref{remark: possible triviality of qms}).
\end{thm}
\begin{proof}
Define $\qs:  C^{\infty}_{cc}(M)\to \RR$ by
\beq\label{equation: sympl-quasi-state zetaq}
\qs(F):=\fr{n}_F+\qm\(\til{\phi^1_F}\).
\eeq
We follow \cite[Theorem $3.1$]{E-P2}. By the Lipschitz property of the quasi-morphism $\qm$, the function  $\qs$ extends continuously to $ C_{cc}(M)$. Since  $\qm\(\til{\id_M}\)=0$, we have $\qs(1)=1$. Symplectic invariance and monotonicity are direct consequences of the same properties of the spectral numbers. The vanishing axiom follows immediately from Corollary~\ref{cor: homogeneous quasi-morphism tilr on hamtil}. For the strong quasi-additivity axiom we note that on the one hand, $\{F,G\}=0$ implies $\til{\phi^1_F}\til{\phi^1_G}=\til{\phi^1_G}\til{\phi^1_F}=\til{\phi^1_{F+G}}$. On the other hand, the restriction of a homogeneous  quasi-morphism to any abelian subgroup is a homomorphism. These two observations directly lead to the result.
\end{proof}

If the quantum homology algebras $QH^\pm_*$ of $(M, \omega)$ do not satisfy conditions $\mathbf{(QM_1)}$ and $\mathbf{(QM_2)}$, Formula ~\eqref{equation: sympl-quasi-state zetaq} defines only a partial  quasi-state (see Theorem~\ref{thm: symplectic partial quasi-state}). In general, the subalgebra $QH^-_{2n}$ does not contain non-zero idempotents. For example, in the weakly exact case there are no quantum effects and the quantum homology algebras $QH^\pm_*$ of $(M, \omega)$ reduce to the usual singular homology algebras with the standard intersection products. On the other hand, the fundamental class $[M,\partial M]\otimes 1$ is always a non-zero idempotent in $QH^+_{2n}$. Therefore,using the partial quasi-morphism $\pqm:=\hat{c^+}$ defined by $\e^+:=[M,\partial M]\otimes 1$, we can construct  a partial symplectic quasi-state on $ C_{cc}(M)$ for any strongly semi-positive compact convex symplectic manifold $(M, \omega)$.
\begin{thm}\label{thm: symplectic partial quasi-state}
Suppose $(M,\omega)$ is a strongly semi-positive compact convex symplectic manifold. Then $ C_{cc}(M)$ admits a partial symplectic quasi-state.
\end{thm}
\begin{rem}
A priori, this partial symplectic quasi-state might be trivial.
\end{rem}
\begin{proof}
Take any non-zero idempotent $\e^+\in QH^+_{2n}$. Define $\pqs: C^{\infty}_{cc}(M)\to\RR$  by
\beq\label{equation: partial quasi-state}
\pqs(F):=\fr{n}_F+\pqm\(\til{\phi^1_F}\).
\eeq
Now, the proof repeats verbatim \cite[ Theorem $4.1$]{E-P2}.
\end{proof}

\subsection{Non-triviality of the quasi-morphism  and of the  quasi-state}\label{section: non-triviality of qms}
A priori, $\qm$ might be a homomorphism.  In this section we present examples of monotone compact convex symplectic manifolds, for which $\qm$ is a genuine homogeneous quasi-morphism and $\qs$ is a genuine non-linear symplectic quasi-state.

We shall examine the case of the one-point symplectic blow-up $(\til{\DD^*X}, \til{X})$ of the cotangent disk bundle $\DD^*X$ relative to the Lagrangian zero section $X$, i.e. the symplectic blow-up of size $\delta$ that corresponds to a symplectic embedding of pairs
$\displaystyle (\BB^{2\dim X}(\delta), \BB^{\dim X}(\delta))\hookrightarrow(\DD^*X, X)$,
where $\BB^{2\dim X}(\delta)$ is the standard closed ball with center zero and radius $(\frac{\delta}{\pi})^{1/2}$ in $\CC^{\dim X}$ and   $\BB^{\dim X}(\delta)$ is its real part in $\RR^{\dim X}=Re\, \CC^{\dim X}$, see \cite[Theorem $1.21$]{R}. Remark that the blown up pair is diffeomorphic to $(\DD^*X\#\cl{\CC\PP^{\dim X}}, X\#\cl{\RR\PP^{\dim X}})$.

Following \cite[Section $8$]{E-P3} we adapt the  Biran-Cornea pearl homology machinery to our setting. Set $\FF=\ZZ_2$. Let $\dim X=2n$ and denote by $E=\CC\PP^{2n-1}$ the exceptional divisor. Let $J\in\CL{J}(M,\del M,\omega)$ be a generic almost complex structure that is standard near $E$. Since $H_2(M,L;\ZZ)=H_2(\CC\PP^{2n},\RR\PP^{2n};\ZZ)\cong\ZZ$, we have that $L$ is the closed monotone Lagrangian submanifold with the minimal Maslov number $N_L =2n-1$, and $M$ is a compact convex monotone symplectic manifold with the minimal Chern number $C_M=2n-1$.
Since $H_2^S(M;\ZZ)=H_2^S(\DD^*X;\ZZ)\oplus\ZZ\la [E]^{2n-1}\ra\cong H_2^S(\DD^*X;\ZZ)\oplus\ZZ$, and $\omega\big|_{H_2^S(\dd^*X;\zz)}=0$,  $\omega([E]^{2n-1})=\delta$,  the group of half-periods of $M$ is $G = \cfrac{\delta}{2}\cdot \ZZ$ and hence $\KK_G=\ZZ_2[[s^{\delta/2}]$ -- the field of formal Laurent series in variable $s^{\delta/2}$ with a finite Taylor part.
The monotonicity constant of $M$ is $\kappa=\frac{\delta}{2n-1}$, and the monotonicity constant of $L$ is $\eta=\frac{\kappa}{2}=\frac{\delta}{4n-2}$. Recall that $\omega = \kappa \cdot c_1$ and $\omega\big|_{H^D_2(M,L;\zz)} = \eta \cdot \mu$, where $H^D_2(M,L;\ZZ)$ is the image of the relative Hurewicz homomorphism $\pi_2(M,L)\to H_2(M,L;\ZZ)$ and $\mu:H^D_2(M,L;\ZZ)\to\ZZ$ is given by the Maslov index.

We also shall use the following notations: $\(H^+_*,\bullet^+\):= \(H_*(M,\partial M;\ZZ_2),\bullet_3\)$ and $\(H^-_*,\bullet^-\):= \(H_*(M;\ZZ_2),\bullet_1\)$.

\begin{defn}\cite{B-C2}.\label{defn: depth and height} Let $L$ be a Lagrangian submanifold of the symplectic manifold $(M,\omega)$ and $\til{\phi}\in\tHam_c(M,\omega)$.
   The \textbf{depth} and, respectively, the \textbf{height} of $\til{\phi}$ on  $L$ are defined as:
   \begin{align*}
      \mathrm{depth}_{L}(\til{\phi}) & =\sup_{[\phi_H^t]=\til{\phi}}\;\inf_{\gamma\in\Lambda(L)}
      \int_{\sss^1}H(\gamma(t),t)dt \\
      \mathrm{height}_{L}(\til{\phi})& =\inf_{[\phi_H^t]=\til{\phi}}\;\sup_{\gamma\in\Lambda(L)}
      \int_{\sss^1}H(\gamma(t),t)dt~,~
   \end{align*}
   where $\Lambda(L)$ stands for the space of smooth loops $\gamma:\SSS^1\to L$, $H:M\times\SSS^1\to \RR$ is a compactly supported Hamiltonian,
   and $[\phi_H^t]=\til{\phi}$ means the equality of homotopy classes of paths (relative to the ends).
\end{defn}

\noindent Our goal is to prove the following
\begin{thm}\label{thm: non-triviality of qms}
Let $X$ be a smooth closed connected manifold of dimension $2n$ for $n>1$, such that
$X$  is a homology sphere over $\ZZ_2$ or the singular homology algebra (w.r.t the intersection product) $H_*(X;\ZZ_2)$ is generated as a ring by $H_{2n-1}(X;\ZZ_2)$. Let $(M, L):=(\til{\DD^*X}, \til{X})$ be the one-point symplectic blow-up relative to the Lagrangian zero section $X$ as above. Then there exist two non-zero idempotents $\e^\pm\in QH^\pm_{4n}(M)$ satisfying  conditions $\mathbf{(QM_1)}$ and $\mathbf{(QM_2)}$, so that
$$\mathrm{depth}_L(\til{\phi})\leq c^\pm(\e^\pm,\til{\phi})\leq \mathrm{height}_L(\til{\phi})+\delta,\ \forall\ \til{\phi}\in\tHam_c(M,\omega).$$
In particular,
\beq\label{equation: depth-height estimate}
\mathrm{depth}_L(\til{\phi})\leq\qm(\til{\phi})\leq \mathrm{height}_L(\til{\phi}),\ \forall\ \til{\phi}\in\tHam_c(M,\omega),
\eeq
where  $\qm(\til{\phi})$ is the homogenization of $c^\pm(\e^\pm,\til{\phi})$.
\end{thm}

\begin{proof}[\textbf{Proof of Theorem \ref{thm:Main theorem on qm}}(assuming Theorem~\ref{thm: non-triviality of qms})] We shall show that the symplectic quasi-state $\qs: C^{\infty}_{cc}(\til{\DD^*X})\to \RR$  given by $\qs(F):=\fr{n}_F+\qm\(\til{\phi^1_F}\)$ is a non-linear functional and thus, $\qm$ is not a homomorphism. Indeed, let $L_1, L_2, L_3$ be three compactly supported symplectic isotopies of $\til{X}$, such that $L_1\cap L_2\cap L_3=\varnothing$ and $L_i\pitchfork L_j$ for all $i, j$. Take
$F_1, F_2\in C_{cc}^{\infty}(\til{\DD^*X})$, such that $F_i|_{L_i}\equiv 0$ for $i=1, 2$ and $(F_1+F_2)|_{L_3}\equiv 1$. By \eqref{equation: depth-height estimate} we have $
\min\limits_{\til{X}}F\leq\qs(F)\leq\max\limits_{\til{X}}F$, for all $F\in C_{cc}^{\infty}(\til{\DD^*X})$, and by the symplectic invariance we have
$\qs(F)=\qs(F\circ\phi)$, for all $F\in C_{cc}^{\infty}(\til{\DD^*X})$, $\phi\in\Symp_c^0(\til{\DD^*X})$.
It follows that $\qs(F_1)=\qs(F_2)=0$ and $\qs(F_1+F_2)=1$.
\end{proof}

\begin{exs}
Examples of $X$ for which $H_*(X;\ZZ_2)$ is generated as a ring by $H_{2n-1}(X;\ZZ_2)$ are direct products of smooth closed 2-dimensional surfaces (orientable or not), homology projective spaces over $\ZZ_2$, and direct products or connected sums of the above manifolds.
\end{exs}

For the proof of the last theorem, we use the theory of the (pearl) Lagrangian quantum homology developed by  Biran and  Cornea, see \cite{B-C1}, \cite{B-C2}, \cite{B-C3}, \cite{B-C4}. Let $QH(L;\Lambda')$ be the  pearl homology of $L$ over the Novikov ring $\Lambda':=\ZZ_2\left[t,t^{-1}\right]$, where $t=s^{-\delta/2}q^{-(2n-1)}$. Note that $\Lambda'$ is the sub-ring of $\Lambda$. Since $M$ is tame, the pearl homology is well-defined. We are interested in the following algebraic operations that have been defined in the above works in the closed case, and readily extend to our case by using admissible Hamiltonians and Morse functions on $M$, see \cite[Theorem $2.5.2$, Theorem $2.6.1$]{B-C3} for the full list of properties.

\subsubsection*{{\bf The module structures $\circledast^\pm$.}}
There are two canonical operations
$$
\circledast^\pm: H_i^\pm\otimes_{\zz_2}QH_j(L;\Lambda')\to QH_{i+j-4n}(L;\Lambda'),\;\;\; \forall\ i,j\in\ZZ
$$
that make $QH(L;\Lambda')$ into a module over the rings $\(H_*^\pm,\bullet^\pm\)$.
\subsubsection*{{\bf The quantum inclusion maps $i^\pm_L$.}}
There is a canonical map of degree zero $\displaystyle i^\pm_L:QH(L;\Lambda')\to QH^\pm$, which is a map of $H_*^\pm$-modules. In addition, using the $\ast_2$ quantum product one can observe that the map $i^-_L$ is also a map of $H_*^+$-modules, i.e. $i^-_L(a\circledast^+x)=a\ast_2 i^-_L(x)$ for $a\in H_*^+$ and $x\in QH(L;\Lambda')$. Note that for any $x\in H_*(L;\ZZ_2)$ the classical part of $i_L^-(x)$ (resp. $i_L^+(x)$) is equal to $i_*(x)\in H_*(M;\ZZ_2)$ (resp. $(j\circ i)_*(x)\in H_*(M, \del M;\ZZ_2)$). Here $i:L\hookrightarrow M$ and $j:(M,\varnothing)\hookrightarrow (M,\del M)$ are the standard inclusions.
%
\subsubsection*{{\bf Gromov-Witten invariants.}}
Invariants $GW_{A,1,3}(a,b,c)$ and $GW_{A,2,3}(a,b,c)$ (over $\ZZ_2$) can be non-zero only when the classes $a,b,c$ come from the classes $[E]^i, ([E]^i)^\vee$ and $A=d[E]^{2n-1}=d[\CC\PP^1]$ for $d\in\NN$.
Since the class $([E]^i)^\vee$ can be represented by a submanifold $\CC\PP^i\subset E$ and  we are using a generic almost complex structure  that is standard near $E$, we have
\beqn
GW_{d[\cc\pp^1],1,3}\([E]^i, ([E]^j)^\vee, ([E]^k)^\vee\)=
\begin{cases}
1&\text{if }\  \begin{aligned}&d=1\text{ and } 0<i, j, k<2n\ \text{with} \\& i-j-k=-1, \end{aligned}\\
0& \text{otherwise},
\end{cases}
\eeq
and
\noindent\beqn
GW_{d[\cc\pp^1],2,3}\([E]^i, [E]^j, ([E]^k)^\vee\)=
\begin{cases}
1&\text{if }\ \begin{aligned}&d=1\text{ and } 0<i, j, k<2n\ \text{with}\\ &i+j-k=2n-1, \end{aligned}\\
0& \text{otherwise}.
\end{cases}
\eeq

\subsubsection*{{\bf Quantum homologies $QH_{4n}^\pm$ and condition $\mathbf{(QM1)}$.}}
Denote $\ve^-_0:=[pt]\otimes q^{4n}, \ve^-_i:=[E]^{2n-i}\otimes q^{4n-2i}$ for $0<i<2n$.
Then $\e^-:=s^\delta\ve^-_1$ is an idempotent w.r.t. the $\ast^-$ product and
$$
\e^-\,QH_{4n}^-=\Span_{\kk_G}\{\ve^-_1, \ve^-_2,\ldots, \ve^-_{2n-2}, s^\delta\ve^-_0+\ve^-_{2n-1}\}\cong\frac{\KK_G[x]}{(x^{2n-1}+s^{-\delta})}
$$
is a field. Indeed, $s^\delta\ve^-_0+\ve^-_{2n-1}\mapsto x$ gives rise to an isomorphism of $\KK_G$-algebras.\\

\noindent Denote $\ve^+_i:=([E]^i)^\vee\otimes q^{4n-2i}$ for $0<i<2n$ and $\ve^+_{2n}:=[M,\del M]$. Then $\e^+:=s^\delta\ve^+_1$ is an idempotent w.r.t. the $\ast^+$ product and
$$
\e^+\,QH_{4n}^+=\Span_{\kk_G}\{\ve^+_1, \ve^+_2,\ldots,\ve^+_{2n-1}\}\cong\frac{\KK_G[x]}{(x^{2n-1}+s^{-\delta})}
$$
is a field. Indeed, $\ve^+_{2n-1}\mapsto x$ gives rise to an isomorphism of $\KK_G$-algebras.
\subsubsection*{{\bf Condition $\mathbf{(QM2)}$.}}
Take $a^-=\sum\limits_{i=0}^{2n-1}f_i(s^{\delta/2})q^{-4n}\ve^-_i+\gamma\in QH_0^-$, where $f_i(s^{\delta/2})\in\KK_G$ and $\gamma$ comes from the homology of $\DD^*X$. If
$$
\Pi_2(\e^+\ast_2 a^-,[M,\del M]\otimes 1)=\(f_{2n-1}(s^{\delta/2})s^\delta\)_0\neq 0,
$$
then $(\e^-\ast^- a^-)_{[pt]}=f_{2n-1}(s^{\delta/2})s^\delta\neq 0$.\\

\noindent Take $a^+=\sum\limits_{i=1}^{2n}f_i(s^{\delta/2})q^{-4n}\ve^+_i+\xi\in QH_0^+$, where $f_i(s^{\delta/2})\in\KK_G$ and $\xi$ comes from the homology of $(\DD^*X, \del \DD^*X)$. If
$$
\Pi_2(\e^-\ast_2 a^+,[M,\del M]\otimes 1)=\(f_{2n-1}(s^{\delta/2})s^\delta\)_0\neq 0,
$$
then $(\e^+\ast^+ a^+)_{q^{-4n}\ve^+_{2n-1}}=f_{2n-1}(s^{\delta/2})s^\delta\neq 0$.\\

\noindent
We conclude that the elements $\e^\pm$ are idempotents needed to define the quasi-morphism $\qm$.
\subsubsection*{{\bf Lagrangian quantum homology.}}
It follows from \cite[Theorem $1.2.2 (i)$]{B-C2} that for any $X$ with the above homological conditions  $QH(L;\Lambda')\cong H(L;\ZZ_2)\otimes_{\zz_2}\Lambda'$. Consider the class $[\RR\PP^2]\in H_2(L;\ZZ_2)=H_2(X;\ZZ_2)\oplus\ZZ_2\la[\RR\PP^2]\ra$. We  have that $i_*([\RR\PP^2])=[E]^{2n-1}=[\CC\PP^1]\in H_2(M;\ZZ_2)$. Since the minimal Maslov number $N_L=2n-1$ is odd  and $j_*([E]^{2n-1})=[E]^\vee$, we have that $i^+_L([\RR\PP^2])=[E]^\vee$ and $i^-_L([\RR\PP^2])=[E]^{2n-1}$. In particular, $\e^\pm=i^\pm_L([\RR\PP^2])\otimes s^\delta q^{4n-2}$. By the degree reason, we see that $QH_2(L;\Lambda')=H_2(L;\ZZ_2)$. It follows that
$$
i^-_L([E]^\vee\circledast^+[\RR\PP^2])=[E]^\vee\ast_2 i^-_L([\RR\PP^2])=[E]^{2n-1}\otimes s^{-\delta}q^{2-4n},
$$
and
$$
i^-_L([E]^{2n-1}\circledast^-[\RR\PP^2])=[E]^{2n-1}\ast^- i^-_L([\RR\PP^2])=[E]^{2n-1}\otimes s^{-\delta}q^{2-4n}.
$$

\begin{lem}
The following holds
$$
[E]^\vee\circledast^+[\RR\PP^2]=[E]^{2n-1}\circledast^-[\RR\PP^2]=
[\RR\PP^2]\otimes s^{-\delta}q^{2-4n}.
$$
\end{lem}
\begin{proof}
Since
$$
[E]^\vee\circledast^+[\RR\PP^2], [E]^{2n-1}\circledast^-[\RR\PP^2]\in QH_{4-4n}(L)=H_2(L;\ZZ_2)\otimes_{\zz_2}s^{-\delta}q^{2-4n}\ZZ_2,
$$
we have
$$
[E]^\vee\circledast^+[\RR\PP^2]=(a+k[\RR\PP^2])\otimes s^{-\delta}q^{2-4n},\;
[E]^{2n-1}\circledast^-[\RR\PP^2]=(b+l[\RR\PP^2])\otimes s^{-\delta}q^{2-4n},
$$
where $a,b\in H_2(X;\ZZ_2)$ and $k,l\in\ZZ_2$. Thus,
$$
(i^-_L(a)+k[E]^{2n-1})\otimes s^{-\delta}q^{2-4n}=[E]^{2n-1}\otimes s^{-\delta}q^{2-4n}$$
and
$$
(i^-_L(b)+l[E]^{2n-1})\otimes s^{-\delta}q^{2-4n}=[E]^{2n-1}\otimes s^{-\delta}q^{2-4n}.
$$
Since the classical (degree zero) part of $i^-_L(a)$ (resp. $i^-_L(b)$) is  $i_*(a)=a$ (resp. $i_*(b)=b$), we have by the degree reason that $i^-_L(a)=a$ and $i^-_L(b)=b$. It follows that  $a=b=0$ and $k=l=1$.
 \end{proof}
\begin{proof}[\textbf{Proof of Theorem\ \ref{thm: non-triviality of qms}}]
The above computations showed that the conditions of Lemma $5.3.1$ from \cite{B-C2} are fulfilled and hence we have the following estimates:\\

\noindent
From \cite[Lemma $5.3.1 (i)$]{B-C2}  we have
$$
c^+([E]^\vee,\til{\phi}), c^-([E]^{2n-1},\til{\phi})\geq \mathrm{depth}_L(\til{\phi})-\delta,\ \forall\ \til{\phi}\in\tHam_c(M,\omega).
$$
Hence, $c^\pm(\e^\pm,\til{\phi})\geq \mathrm{depth}_L(\til{\phi}),\ \forall\ \til{\phi}\in\tHam_c(M,\omega)$, and hence
$$\qm(\til{\phi})\geq \mathrm{depth}_L(\til{\phi}),\ \forall\, \til{\phi}\in\tHam_c(M,\omega).$$
\noindent
From \cite[Lemma $5.3.1 (ii)$]{B-C2} we have
$$c^\pm(\e^\pm,\til{\phi})\leq \mathrm{height}_L(\til{\phi})+\delta,\ \forall\ \til{\phi}\in\tHam_c(M,\omega),$$ and hence
$\qm(\til{\phi})\leq \mathrm{height}_L(\til{\phi}),\ \forall\, \til{\phi}\in\tHam_c(M,\omega).$
\end{proof}

\begin{rem}
 It might seem that the  simplest manifold for which one should try to check  the non-triviality of the quasi-morphism/quasi-state is  the symplectic blow-up $(\til{\BB^4} , \omega_{\delta})$ of  size  $0<\delta<<1$ at zero of the standard closed unit $4$-ball $\BB^4$ in $(\CC^2,\omega_0)$. Indeed, it is a  compact convex symplectic manifold, which satisfies conditions $\mathbf{(QM_1)}$ and $\mathbf{(QM_2)}$ of Theorem \ref{thm: quasi-morphism q on hamtil}
Hence, the group $\tHam_c(\til{\BB^4} , \omega_{\delta})$ admits a homogeneous quasi-morphism $\qm$ as above. In addition, $\til{\BB^4}$ contains a monotone closed Lagrangian submanifold $L$ of minimal Maslov number $N_L=2$  -- the Clifford torus $\TT^2_{clif}$. So one can hope to get  similar estimates for $\qm$  using the Lagrangian quantum homology $QH(\TT^2_{clif})$ of $\TT^2_{clif}$ as before. But unfortunately, from the geometric criterion for the vanishing of the Lagrangian quantum homology, see \cite[Proposition 4.2.1]{B-C2}, the algebra $QH(\TT^2_{clif})$ vanishes. Hence, the question of the non-triviality of  $\qm$ for $(\til{\BB^4} , \omega_{\delta})$ is open.
\end{rem}

\subsection{Descent of $\qm$  to the group $\Ham_c(M,\omega)$ }
In this section we shall prove Theorem~\ref{thm:restriction of qm to hamc}. To this end, we shall show that the restriction of $\qm$  to the abelian subgroup $\pi_1\Ham_c(M,\omega) \subset \tHam_c(M,\omega)$ vanishes identically. For this purpose, we follow \cite{E-P1} and use  the Seidel homomorphism $\Psi: \pi_1\Ham_c(M,\omega) \to \(QH^+_*, \ast^+\)^{\times}$ (see \cite{Seidel}, cf. \cite{LMP}), where
$\(QH^+_*, \ast^+\)^{\times}$ denotes the group of units of the relative quantum homology algebra $\(QH^+_*, \ast^+\)$ of $M$.
We follow~\cite{LMP} in the description of the Seidel homomorphism in the case of compact convex strongly semi-positive symplectic manifolds.

So, let  $\varphi$ be a loop of Hamiltonian diffeomorphisms  and let
$$
(M,\omega) \hookrightarrow (P_{\varphi},\omega_{\varphi}) \overset{\pi}{\twoheadrightarrow} \SSS^2
$$
be a Hamiltonian fibration over the two-sphere $\SSS^2$ with the coupling class $u_{\varphi}$ and the vertical Chern  class $c_{\varphi}$. We define an equivalence relation on the space of sections of the fibration $P_{\varphi} \overset{\pi}{\twoheadrightarrow}\SSS^2$ in the following way: two sections $\sigma_1$ and $\sigma_2$ are said to be $\Gamma$-{\it equivalent} if
$
u_{\varphi}[\sigma_1(\SSS^2)] = u_{\varphi}[\sigma_2(\SSS^2)], \ \
c_{\varphi}[\sigma_1(\SSS^2)] = c_{\varphi}[\sigma_2(\SSS^2)].
$
It has been shown in~\cite{Seidel} that the set $\CL{S}_{\varphi}$ of all such equivalent classes is an affine space modeled on the  group $\Gamma$, see~\eqref{equation: Gamma group of spherical classes}.

Next, we fix the canonical section class $[\sigma]_{\varphi}$ that is
uniquely determined by
$
u_{\varphi}([\sigma]_{\varphi})=c_{\varphi}([\sigma]_{\varphi}) = 0
$
and define a group homomorphism
$$ \rho: \pi_1\Ham_c(M,\omega) \to \End_\Lambda(QH^+_*) \text{ by } \rho([\varphi])=\Psi_{\varphi,[\sigma]_{\varphi}},$$
where
$$ \Psi_{\varphi,[\sigma]_{\varphi}}(a) = \sum_{[B]\in \Gamma}  a_{[B]}\otimes s^{-\omega([B])}q^{-2c_1([B])},\ a_{[B]}
\in H_{\deg(a) + 2c_{1}([B])}(M, \del M;\ZZ_2).$$
The class $a_{[B]}$  is uniquely defined by  $a_{[B]} \bullet_2 b = \sum\limits_{B'\in [B]}SGW_{B',1,2}(a,b)$, where  $SGW_{A,1,2}(a,b)$ is the $2$-pointed sectional Gromov-Witten invariant  -- see Section \ref{subsection: QH}.
Like in the closed case (see~\cite{Seidel} or~\cite{LMP})  $\Psi_{\varphi,[\sigma]_{\varphi}}$ is an isomorphism for any loop $\varphi$ and actually depends only on $[\varphi] \in\pi_1\Ham_c(M,\omega)$. Moreover,  for any $[\varphi] \in\pi_1\Ham_c(M,\omega)$ we have $ \rho([\varphi])(a) =
\Psi_{\varphi,[\sigma]_{\varphi}}([M,\del M])\ast^+ a$. The Seidel
homomorphism $\Psi: \pi_1\Ham_c(M,\omega) \to \(QH^+_*, \ast^+\)^{\times}$ is defined by $\varphi \mapsto \rho([\varphi])([M, \del M])$. By the same type of arguments as in the closed case one can show that this is indeed a well-defined homomorphism -- see \cite{Seidel}, \cite{LMP}, \cite{MS3} for more details. We have the following
\begin{prop}
Let $M$ be one of the manifolds listed in Theorem~\ref{thm: non-triviality of qms}. Then $\Psi([\varphi])=[M,\del M]\otimes 1$, for any $[\varphi]\in\pi_1\Ham_c(M,\omega)$.
\end{prop}

\begin{proof}
Recall that for such an $M$ we have that $\Gamma=\ZZ\la [\CC\PP^1]\ra$,  $\omega([\CC\PP^1])=\delta$ and $c_1([\CC\PP^1])=2n-1$, where $[\CC\PP^1]=[E]^{2n-1}$ and $E=\CC\PP^{2n-1}$ is the exceptional divisor. We need to show that $\Psi_{\varphi,\sigma_{\varphi}}([M,\del M])=[M,\del M]\otimes 1$ for any $[\varphi]\in\pi_1\Ham_c(M,\omega)$.

First, note that since $GW_{A,1,3}(a,b,[pt])=0$ for all $A\in H_2^S(M)$ and all $a, b\in  H_*(M,\partial M;\ZZ_2)$, it follows that the $\Lambda$-module $$I:=\bigoplus_{i<4n}H_i(M,\partial M;\ZZ_2)\otimes_{\zz_2}\Lambda$$ is an ideal in $\(QH^+_*, \ast^+\)$. (Here $4n=\dim M$, $n>1$.) Then any unit in $\(QH^+_*, \ast^+\)$ is of the form $[M,\del M]\otimes\lambda+x$, where $x\in I$ and $\lambda\in\Lambda$ is a unit. This statement is analogous to  \cite[Lemma 2.1]{McDuff_Uniruled}.

Second, we have that
$$
\Psi_{\varphi,[\sigma]_{\varphi}}([M,\del M]) = \sum_{k\in\zz}  a_{k[\cc\pp^1]}\otimes s^{-k\delta}q^{-2k(2n-1)},
$$
where $a_{k[\cc\pp^1]}\in H_{4n + 2k(2n-1)}(M, \del M;\ZZ_2)$. By the form of a unit in $\(QH^+_*, \ast^+\)$, we have that
$a_{0}= [M,\del M]$. Moreover, since $H_{4n + 2k(2n-1)}(M, \del M;\ZZ_2)=0$ for all $k\neq 0, -1$,  we have that
$$
\Psi_{\varphi,[\sigma]_{\varphi}}([M,\del M]) = [M,\del M]\otimes 1+  a_{-[\cc\pp^1]}\otimes s^{\delta}q^{4n-2},
$$
where $a_{-[\cc\pp^1]}\in H_2(M, \del M;\ZZ_2)$. But since
$\dim X<4n-2$ we have
\begin{multline*}
H_2(M, \del M;\ZZ_2)\cong H_2(\DD^*X,\del \DD^*X;\ZZ_2)\oplus \ZZ_2\la [E]^\vee\ra\cong H_{4n-2}(\DD^*X;\ZZ_2)\oplus \ZZ_2\la [E]^\vee\ra\\ \cong H_{4n-2}(X;\ZZ_2)\oplus \ZZ_2\la [E]^\vee\ra\cong\ZZ_2\la [E]^\vee\ra,
\end{multline*}
and we conclude that
$$
\Psi_{\varphi,[\sigma]_{\varphi}}([M,\del M])=[M,\del M]\otimes 1+a[E]^\vee\otimes s^{\delta}q^{4n-2},
$$
where $a\in\ZZ_2$.

Now, suppose on the contrary that $a =1$. Then
$$
\big([M,\del M]\otimes 1+[E]^\vee\otimes s^{\delta}q^{4n-2}\big)\ast^+\big([M,\del M]\otimes 1+\sum\limits_{i=1}^{2n-1}([E]^i)^\vee\otimes\mu_i+x\big)=\ [M,\del M]\otimes 1
$$
for some $\mu_i\in\Lambda$ and $x\in\bigoplus_{i<4n}H_i(\DD^*X,\del \DD^*X;\ZZ_2)\otimes_{\zz_2}\Lambda$. An easy computation shows that
\begin{multline*}
\big([M,\del M]\otimes 1+[E]^\vee\otimes s^{\delta}q^{4n-2}\big)\ast^+\big([M,\del M]\otimes 1+\sum\limits_{i=1}^{2n-1}([E]^i)^\vee\otimes\mu_i+x\big)=\\ [M,\del M]\otimes 1+x+[E]^\vee\otimes s^{\delta}q^{4n-2}.
\end{multline*}
The element $(x+[E]^\vee\otimes s^{\delta}q^{4n-2})\in I$ differs from zero, because if  $x\neq 0$ then $x$ and $[E]^\vee\otimes s^{\delta}q^{4n-2}$ are $\Lambda$-linearly independent. Thus we get a contradiction and the proposition is proven.
\end{proof}

Repeating the proof of \cite[Lemma 12.5.3]{MS3} (see also \cite[Proposition 4.1]{E-P1}) we get the following

\begin{prop}\label{prop: spectral c+ on loops}
Let $M$ be one of the manifolds listed in Theorem~\ref{thm: non-triviality of qms}. For every class $ [\varphi ] \in \pi_1\Ham_c(M,\omega)$ and every non-zero $a \in QH^+_*$ we have $$
c^+(a,[\varphi]) = \nu(a\ast^+\Psi([{\varphi}])^{-1})=\nu(a).$$
\end{prop}

\begin{proof}[\textbf{Proof of  Theorem~\ref{thm:restriction of qm to hamc}}]
By Proposition \ref{prop: spectral c+ on loops}, we have that $$c^+([\varphi])=c^+(\e^+,[\varphi])=\nu(\e^+)=\delta$$
for any class $ [\varphi ] \in \pi_1\Ham_c(M,\omega)$,
where $\e^+=[E]^\vee\otimes s^\delta q^{4n-2}\in QH^+_{4n}$. Thus, the homogenization $\qm$ of $c^+$ vanishes on $\pi_1\Ham_c(M,\omega)$.
\end{proof}

\subsection{Non-triviality of $\pqm$  and $\pqs$ for weakly exact manifolds}\label{subsection:aspherical manifolds}
Set $\FF=\ZZ_2$. Let $(M,\omega)$ be a convex compact weakly exact symplectic manifold, i.e. $[\omega]|_{\pi_2(M)}=0$, and let $L\subset M\minus\del M$ be a closed Lagrangian submanifold such that
\be
\item[1.]the inclusion map $L\hookrightarrow M$ induces an injection $\pi_1(L)\hookrightarrow\pi_1(M)$,
\item[2.] $L$ admits a Riemannian metric with no non-constant contractible closed geodesics.
\ee
Since $[\omega]|_{\pi_2(M)}=0$  the quantum homologies $(QH^\pm_*,\ast^\pm)$ are undeformed, i.e. the quantum products $\ast^\pm$ coincide with the classical intersection products $\bullet^\pm$. In particular, there is no non-zero idempotent in the subalgebra $QH^-_{2n}$. On the other hand, the fundamental class $\e^+:=[M,\partial M]\otimes 1$ is  a non-zero idempotent in $QH^+_{2n}$, which defines the partial quasi-morphism $\pqm:=\hat{c^+}:\tHam_c(M,\omega)\to\RR$ -- the homogenization of $c^+\(\til{\phi}\):=c^+\(\e^+,\til{\phi}\)$, for  $\til{\phi}\in\tHam_c(M, \omega)$.  Hence, the functional  $\pqs: C_{cc}(M)\to\RR$ given by  $\pqs(F):=\fr{n}_F+\pqm\(\til{\phi^1_F}\)$ is a partial symplectic quasi-state, see Theorem \ref{thm: symplectic partial quasi-state}.

We note that the spectral invariant $c^+\(\e^+,\til{\phi}\)$ was defined in \cite{F-S} for the class of convex compact  weakly exact symplectic manifolds and is also called an action selector. It has the following properties:
\be
\item{}The invariant $c^+$ vanishes on $\pi_1\tHam_c(M,\omega)$, see \cite[Proposition 7.1]{F-S}.

\item{}For $F\in\CL{S}_{HZ}^{\circ}(M)$ we have $c^+\(\e^+,\phi^1_F\)=\max_MF$, see \cite[Theorem 5.3]{F-S}, where $\CL{S}_{HZ}^{\circ}(M)$ is the class of \emph{\textbf{Hofer-Zehnder admissible}} Hamiltonians.
\ee
\noindent Recall that $F\in\CL{S}_{HZ}^{\circ}(M)\subset \CL{H}_c(M)\cap C^\infty_c(M)$ iff
\be
\item[$(i)$] $F\geq 0$,
\item[$(ii)$] $F|_U=\max_M F$ for some open non-empty set $U\subset M$,
\item[$(iii)$] the only critical values of $F$ are $0$ and $\max_M F$,
\item[$(iv)$] the flow $\phi^t_F$ has no non-constant, contractible in $M$, $T$-periodic orbits with period $T\leq 1$.
\ee
\noindent
\begin{proof}[\bf{Proof of Proposition~\ref{prop:nontrivial pqs for weakly exact mnflds} (A)}]
Fix a Riemannian metric $g$ on $L$ all of whose contractible closed geodesics are constant. By Weinstein's tubular neighborhood theorem, there exists $\varepsilon>0$ such that a closed $\varepsilon$-neighborhood $U_{\varepsilon}$ of $L$ in $M$ is symplectomorphic to $\DD^*L$ -- the unit cotangent closed disk bundle of $L$ w.r.t. the metric $g$. We choose coordinates $(q,p)$ on $\DD^*L$ and consider a Hamiltonian $F\in\CL{S}_{HZ}^{\circ}(M)$, which is zero on $M\minus U_{\varepsilon}$ and  $F(q,p)=f(\|p\|^2)$ on $U_{\varepsilon}\cong\DD^*L$, where $f\in C^\infty_c([-1,1])$ with $\mathrm{supp}(f)\subset(-1,1)$ and $f(r)\geq 0$ for all $r\in[-1,1]$.  For any $k\in\NN$ the Hamiltonian $kF$  satisfies the conditions $(i)-(iii)$ and $\max_M(kF)=k\max_MF$. Moreover, by the conditions 1-2 on $L$,  the flow $\phi^t_{kF}$ has no non-constant contractible periodic orbits of any period. Hence, $kF\in\CL{S}_{HZ}^{\circ}(M)$ for any $k\in\NN$, and  $c^+\(\e^+,\phi^1_{kF}\)=k\max_MF$.
Now, take $F, G\in\CL{S}_{HZ}^{\circ}(M)$ of the form $F(q,p)=f(\|p\|^2)$ and $G(q,p)=g(\|p\|^2)$ as above, such that $\max_M(F+G)\neq\max_M F+\max_MG$ and such that $supp(F)\cap supp(G)=\varnothing$. We have that $\{F,G\}=0$ (since $supp(\{F,G\})\subseteq supp(F)\cap supp(G)$) and moreover,
\begin{multline*}
\pqs(F+G)=\lim_{k\to+\infty}\frac{c^+\(\e^+,\phi^1_{k(F+G)}\)}{k}=
\lim_{k\to+\infty}\frac{k\max_M(F+G)}{k}=\max_M(F+G)\neq\\
\max_M F+\max_MG=
\lim_{k\to+\infty}\frac{k\max_MF}{k}+\lim_{k\to+\infty}\frac{k\max_M G}{k}= \lim_{k\to+\infty}\frac{c^+\(\e^+,\phi^1_{kF}\)}{k}+\\
\lim_{k\to+\infty}\frac{c^+\(\e^+,\phi^1_{kG}\)}{k}=\pqs(F)+\pqs(G)
\end{multline*}
It follows that $\pqs$ does not satisfy the  strong quasi-linearity property (see Definition~\ref{Def:symplectic quasi-state}), and hence is not a symplectic quasi-state. In particular, $\pqm$ is not a homogeneous quasi-morphism (and, in particular, not a homomorphism).
\end{proof}

\subsection{(Partial) quasi-morphism  and (partial) quasi-state for open convex manifolds}\label{section: qms for exhaustion}
 Recall  that a manifold is called \textit{\textbf{open}} if it is non-compact as a topological space and without boundary.
 An open symplectic manifold $(M,\omega)$ is called  convex if there exists an increasing sequence of compact convex symplectic submanifolds $M_k\subset M$ exhausting $M$, that is, $$M_1\subset M_2\subset\ldots\subset M_k\subset\ldots\subset M\;\;\;\text{and}\;\;\;\bigcup\limits_{k}M_k=M.$$
Frauenfelder, Ginzburg and Schlenk \cite{F-G-S} asked the natural question of patching of spectral invariants defined for $M_k$ into one spectral invariant for $M$.

\begin{cor}\label{cor: pqm for exhaustion}
Let $(M,\omega)$ be an open convex symplectic manifold exhausted by compact convex strongly semi-positive symplectic submanifolds $M_k\subset M,\ k\in\NN$. Suppose that spectral invariants  $c^+_k:\tHam_c(M_k,\omega)\to\RR$ given by
$c^+_k(\til{\phi}):=c^+([M_k,\partial M_k]\otimes 1, \til{\phi})$ satisfy $c^+_{k+1}|_{\widetilde{\Ham}_c(M_k,\omega)}=c^+_k$ for all $k\in\NN$. Then  the corresponding partial quasi-morphisms $\qm_p$ (resp. partial quasi-states $\pqs$) for $M_k$ fit together to a partial quasi-morphism $\qm_p$ (resp. partial quasi-state $\pqs$) for $M$.
\end{cor}

For example, let $M=T^*X$ be the cotangent bundle (over a closed connected smooth manifold $X$) exhausted by cotangent closed $k$-disk bundles $M_k:=\DD_k^*X,\ k\in\NN$ or let $(M, J, f)$ be an open Stein manifold  exhausted by Stein domains $M_k:=\{f\leq k\},\ k\in\NN$, see Example~\ref{example: convex symplectic manifolds}(3). In either cases, $M_k,\ k\in\NN$ are convex compact  weakly exact (in particular,strongly semi-positive) symplectic manifolds. As it was pointed in \cite[Appendix A.3]{F-G-S}, these exhaustions  satisfy the assumption of Corollary~\ref{cor: pqm for exhaustion}.

\begin{proof}[\bf{Proof of Proposition~\ref{prop:nontrivial pqs for weakly exact mnflds} (descent to $\Ham_c(M,\omega)$)}]
The statement follows immediately from the vanishing of the invariant $c^+_k$ on $\pi_1\tHam_c(M_k,\omega)$ for any $k\in\NN$ ( see Section~\ref{subsection:aspherical manifolds}) and from the above patching property.
\end{proof}

\begin{rem}
The argument  in \cite[Appendix A.3]{F-G-S} is readily applicable to cotangent bundles blown up at a point. In particular, let $(M, L):=(\til{T^*X}, \til{X})$ be the one-point symplectic blow-up of size $\delta>0$ relative to the Lagrangian zero section $X$. Recall that such a blow-up  corresponds to a symplectic embedding of pairs $\displaystyle (\BB^{2\dim X}(\delta), \BB^{\dim X}(\delta))\hookrightarrow(T^*X, X)$.  Suppose that the size $\delta$ is sufficiently small so that the above embedding lies in the interior of the cotangent closed unit disk bundle $\DD_1^*X$. The same symplectic embedding gives rise to the relative one-point symplectic blow-up $(M_k, L):=(\til{\DD_k^*X} ,\til{X})$ of the cotangent closed $k$-disk bundle $\DD_k^*X,\; k\in\NN$. Let $X$ be  as in Theorem \ref{thm: non-triviality of qms}. Then the  homogeneous quasi-morphisms $\qm$ (resp. quasi-states $\qs$) for $M_k$ fit together to a quasi-morphism $\qm$ (resp. quasi-state $\qs$) for $M$.
\end{rem}

Finally, let us note that Monzner, Vichery and Zapolsky \cite{M-V-Z} independently constructed partial quasi-morphisms and quasi-states for $T^*X$ using a different construction in the same spirit.
They used Lagrangian spectral invariants to construct, for {\it any} closed connected smooth manifold X and any $a\in H^1(X;\RR)$, a partial quasi-morphism $\mu_a:\Ham_c(T^*X,\omega_{can})\to\RR$ and a partial quasi-state $\zeta_a:C^{\infty}_{cc}(T^*X)\to\RR$, such that $\mu_a\leq\pqm$, $\zeta_a\leq\pqs$ for any $a\in H^1(X;\RR)$. As it follows from their results,  if $X$ admits a non-vanishing closed 1-form then $\pqs(F)=0$ for any $F\in C^{\infty}_{cc}(T^*X)$, such that $F\leq 0, \fr{n}_F=0$. I thank Frol Zapolsky for explaining this point to me.

\begin{proof}[\bf{Proof of Proposition~\ref{prop:nontrivial pqs for weakly exact mnflds} (B)}]
Let $F\in C^{\infty}_{c}(T^*X)$ be a compactly supported non-positive Hamiltonian, such that $F|_X=c<0$. Then $\pqs(F)=0$  and by  \cite[Theorem 1.8(vi)]{M-V-Z} for $a=0\in H^1(X;\RR)$, we have $\zeta_0(-F)=-c>0$. Hence, $\pqs(F)+\pqs(-F)=\pqs(-F)\geq\zeta_0(-F)=-c>0=\pqs(F-F)$.
It follows that $\pqs$ does not satisfy the strong quasi-linearity property (see Definition~\ref{Def:symplectic quasi-state}), and hence is not a symplectic quasi-state. In particular, $\pqm$ is not a homogeneous quasi-morphism (and, in particular, not a homomorphism).
\end{proof}

\section{Applications}\label{section: application}
\subsection{Hofer geometry.}
Recall (see \cite{H}, \cite{Polt}, \cite{LM}) that the \emph{\textbf{ Hofer metric}} $d$ on $\Ham_c(M,\omega)$ is defined by  $d(\phi,\psi):=\|\phi^{-1}\circ\psi\|$,
 where the  \emph{\textbf{Hofer norm}} on $\Ham_c(M,\omega)$ is given by $\|\phi\|:=\inf\{\hn{H} \ |\ \phi^1_H=\phi \}.$ This metric is bi-invariant.

\begin{prop}\label{prop: linear growth of one-parametric subgroup}
Let $(M,\omega)$  be a one-point symplectic blow-up $(\til{\DD^*X},\omega_{\delta})$ of size $\delta<<1$ of $\DD^*X$ relative to $X$, where $X$  is a homology sphere over $\ZZ_2$ or the singular homology algebra (w.r.t the intersection product) $H_*(X;\ZZ_2)$ is generated as a ring by $H_{2n-1}(X;\ZZ_2)$. Then for a time-independent Hamiltonian $F\in\CL{H}_c(M)$, such that $\inf_{\til{X}}F>0$, the one-parametric subgroup $\{\phi^t_F\}_{t\in\rr}$  has a linear growth w.r.t. the Hofer  metric.
\end{prop}
\begin{proof}
Let $\qm:\Ham_c(M,\omega)\to\RR$ be the  homogeneous quasi-morphism from Theorem \ref{thm:restriction of qm to hamc}, and let $\til{X}\subset M$ denote the blown up zero section. By Corollary \ref{cor: homogeneous quasi-morphism tilr on hamtil}, we have  $|\qm(\phi)|\leq d(\id,\phi)$ for all $\phi\in\Ham_c(M,\omega)$. For a time-independent Hamiltonian $F\in\CL{H}_c(M)$, such that $\inf_{\til{X}}F>0$, we have $\qm(\phi^t_F)\geq t\inf_{\til{X}}F$ for all $t\geq 0$ (see Theorem \ref{thm: non-triviality of qms}). Then, by \cite[Formula 1.22]{E-P1}
$$
\lim\limits_{t\to+\infty}\frac{d(\id,\phi^t_F)}{t\hn{F}}\geq\frac{\inf_{\til{X}}F}{\hn{F}}>0.
$$
\end{proof}

\begin{cor}\label{cor: infinite Hofer diameter}
The diameter of\ $\Ham_c(M,\omega)$ w.r.t. the Hofer metric is infinite.
\end{cor}


\subsection{Fragmentation length.}
Let $\KL{U}=\{U_i\}_{i\in I}$ be an open cover of a connected symplectic manifold $(M,\omega)$ and let $\phi\in\Ham_c(M,\omega)$. By Banyaga's fragmentation lemma \cite{B},  $\phi$ can be written as $\phi=\phi_1\circ\ldots\circ\phi_N$,  where each $\phi_i$ lies in $\Ham_c(U_{j(i)})$ for some $j(i)\in I$. Denote by $\|\phi\|_{\Kl{U}}$ the minimal number  of  such $\phi_i$'s needed to factorize $\phi$.

Suppose now that $M:=\til{\DD^*X}$  is a one-point symplectic blow-up of $\DD^*X$ relative to $X$, where $X$  is a homology sphere over $\ZZ_2$ or the singular homology algebra (w.r.t the intersection product) $H_*(X;\ZZ_2)$ is generated as a ring by $H_{2n-1}(X;\ZZ_2)$. Let $\qm:\Ham_c(M,\omega)\to\RR$ be the  homogeneous quasi-morphism as above. We have the following

\begin{cor}\label{cor: fragmentation norm estimate}
Let  $\KL{U}=\{U_i\}_{i\in I}$ be an open cover of $M\minus\del M$, such that each $U_i$ is displaceable by an element of $\Ham_c(M,\omega)$. Then for every $\phi\in\Ham_c(M,\omega)$ we have
$
\|\phi\|_{\Kl{U}}\geq\cfrac{|\qm(\phi)|}{D(\qm)}.
$
\end{cor}

\begin{proof}
Let $\phi\in\Ham_c(M,\omega)$ and let $N:=\|\phi\|_{\Kl{U}}$. Write $\phi=\phi_1\circ\ldots\circ\phi_N$,  where each $\phi_i$ lies in $\Ham_c(U_{j(i)})$ for some $j(i)\in I$.  Since $\qm$ vanishes on every $\Ham_c(U_i)$, we have
$$
|\qm(\phi)|=\Big|\qm(\phi_1\circ\ldots\circ\phi_N)-\sum_{i=1}^N\qm(\phi_i)\Big|\leq (N-1)\cdot D(\qm)\leq N\cdot D(\qm).
$$
It follows that $N\geq\cfrac{|\qm(\phi)|}{D(\qm)}$.
\end{proof}

In the case of a weakly exact  $(M,\omega)$ as in Section \ref{subsection:aspherical manifolds}, we can estimate from below  the fragmentation length $\|\phi\|_U$ w.r.t. a displaceable open set $U\subset M$ (see Definition \ref{def:fragmentation length}) in terms of the partial quasi-morphism $\pqm:\Ham_c(M,\omega)\to\RR$.

\begin{cor}\label{cor: aspherical fragmentation norm estimate}
For every displaceable open set $U\subset M$, there exists a positive constant $K=K(U)$, such that  $\|\phi\|_U\geq K\cdot |\pqm(\phi)|$ for every $\phi\in\Ham_c(M,\omega)$. In particular, the fragmentation length $\|\phi\|_U$ is unbounded.
\end{cor}

\begin{proof}
Since $\phi(U)$ is a displaceable open set,  we have $\pqm(\phi\theta\phi^{-1})=0$ (see  Corollary \ref{cor: homogeneous quasi-morphism tilr on hamtil}). By Theorem \ref{thm: partial quasimorphism r} and repeating the above argument, there exists a positive constant $R=R(U)$, such that
$\displaystyle
|\pqm(\phi)|\leq R\cdot \|\phi\|_U.
$
Setting $K:=1/R$ we get the estimate. For Hofer-Zehnder admissible Hamiltonians $0\neq F\in\CL{S}_{HZ}^{\circ}(M)$ as in the proof of Proposition
\ref{prop:nontrivial pqs for weakly exact mnflds} {\bf(A)}, we have
$$
\lim_{k\to+\infty}|\pqm(\phi^1_{kF})|=\lim_{k\to+\infty}k\max_MF=+\infty.
$$
Hence, $\lim_{k\to+\infty}\|\phi^1_{kF}\|_U=+\infty$.

\end{proof}

\subsection{Symplectic rigidity.}
Here we extend several notions related  to rigidity of intersections in symplectic manifolds that were introduced by Entov-Polterovich in \cite{E-P3} to the case of a compact convex symplectic manifold.

\begin{defn}(See~\cite{E-P3}.)
Let $\pqs: C_{cc}^{\infty}(M) \to \RR$ be a partial symplectic quasi-state. A compact subset $X \subset \overset{\circ}{M}$ is called \textbf{$\pqs$-heavy}  if for every $F\in C_{cc}^{\infty}(M)$ with $F|_X=0$, $F \leq 0$, one has $\pqs(F)=0$, and it is called \textbf{$\pqs$-superheavy}  if for every $F \in C_{cc}^{\infty}(M)$ with $F|_X=0$, $F \geq 0$, one has $\pqs(F)=0$.
\end{defn}

We shall need the following formal properties of $\pqs$-(super)heavy sets that follow immediately from
\cite[Sections 4,6 ]{E-P3}.
\be

\item{} The classes of  $\pqs$-superheavy and $\pqs$-heavy sets are  $\Symp^0_c(M,\omega)$-invariant.

\item{}Every $\pqs$-superheavy set is $\pqs$-heavy (but in general not vice versa). If $\pqs$ is a quasi-state then the classes of  $\pqs$-heavy and $\pqs$-superheavy sets coincide.
\item{} Every $\pqs$-superheavy set intersects with every $\pqs$-heavy set. In particular, every $\pqs$-superheavy set  cannot be displaced by a symplectomorphism from $\Symp^0_c(M,\omega)$.
\item{}Every $\pqs$-heavy subset is non-displaceable by any compactly supported Hamiltonian diffeomorphism of $M$.
\ee
Let  $\qs: C^{\infty}_{cc}(\til{\DD^*X})\to \RR$ be the symplectic quasi-state  given by $\qs(F):=\fr{n}_F+\qm\(\til{\phi^1_F}\)$. By \eqref{equation: depth-height estimate}, we have
$\min\limits_{\til{X}}F\leq\qs(F)\leq\max\limits_{\til{X}}F,$ for all $F\in C_{cc}^{\infty}(\til{\DD^*X})$. Thus, we get the following

\begin{cor}\label{cor: zeta-superheavy}
Let $X$ be a smooth closed connected manifold of dimension $2n$ for $n>1$, such that
$X$  is a homology sphere over $\ZZ_2$ or the singular homology algebra (w.r.t the intersection product) $H_*(X;\ZZ_2)$ is generated as a ring by $H_{2n-1}(X;\ZZ_2)$.
The Lagrangian $\til{X}$ in $(\til{\DD^*X},\omega_{\delta})$ is $\qs$-superheavy. In particular, $\til{X}$ is non-displaceable by a symplectomorphism from $\Symp^0_c(\til{\DD^*X},\omega_{\delta})$.
\end{cor}

\begin{rem}
Note that if the Euler characteristic $\chi(\til{X})$ of $\til{X}$ is zero, i.e. $\chi(X)=1$, then the Lagrangian $\til{X}$ can be displaced from itself topologically. Let us list some examples of such manifolds $X$: $\RR\PP^{2n}$, $(\SSS^2\times\TT^2)\#\RR\PP^4$, $(\SSS^2\times\SSS^2)\#\(\#^{3}_{i=1}\RR\PP^{4}\)$, $(\Sigma_{g_1}\times\Sigma_{g_2})\#\(\#^{(2g_1-2)(2g_2-2)-1}_{i=1}\RR\PP^{4}\)$, where $\Sigma_{g_i}$ is a closed orientable surface of genus $g_i>1$ for $i=1, 2$,  and direct products of the above manifolds.
\end{rem}

In the case of a symplectically aspherical (i.e. $[\omega]|_{\pi_2(M)}=c_1|_{\pi_2(M)}=0$)  connected $(M,\omega)$ with a connected Lagrangian submanifold $L$ as in Section \ref{subsection:aspherical manifolds}, we can look for (super)heavy sets w.r.t. the non-linear partial symplectic quasi-state $\pqs: C_{cc}(M)\to\RR$ from Proposition~\ref{prop:nontrivial pqs for weakly exact mnflds} {\bf(A)}. We have the following

\begin{prop}\label{prop:heavy and not superheavy zero section torus}
The connected Lagrangian submanifold $L$ of $M$ is a $\pqs$-heavy set, but it is not a $\pqs$-superheavy set. (In particular, we recover the result by Frauenfelder and Schlenk \cite{F-S} that $L$  is non-displaceable by any compactly supported Hamiltonian diffeomorphism of $M$.)
\end{prop}

\noindent The statement about the $\pqs$-heaviness is analogous to \cite[Theorem 1.17]{E-P3}, and is based on the spectral estimates, which originated in the works by P. Albers \cite{Alb}, P. Biran and O. Cornea    \cite{B-C1}, \cite{B-C2}, \cite{B-C3}, \cite{B-C4}. As was pointed out in Section \ref{section: non-triviality of qms}, the theory of  pearl homology readily extends to a monotone convex symplectic manifold with a closed monotone Lagrangian submanifold. In the case of a symplectically aspherical convex symplectic manifold  all quantum algebraic structures reduce to the classical homological operations.

\begin{proof}
The quantum inclusion map $\displaystyle i^-_L:QH(L)\to QH(M)$ reduces to the classical  map of degree zero $i_*: H_*(L;\ZZ_2)\to H_*(M;\ZZ_2)$, which is induced by the standard inclusion $i:L\hookrightarrow M$. Moreover, the homomorphism $i_*: H_0(L;\ZZ_2)=\ZZ_2\langle[pt]\rangle\to H_0(M;\ZZ_2)=\ZZ_2\langle[pt]\rangle$ is an isomorphism, i.e. $i_*([pt])=[pt]$ and hence, by  \cite[Lemma $5.3.1\, (ii)$]{B-C2}, we have
$$
c^-([pt],\til{\phi})\leq \mathrm{height}_L(\til{\phi}),\ \forall\ \til{\phi}\in\tHam_c(M,\omega).
$$
In particular, $c^-([pt],H)\leq\sup_LH$ for all $H\in C^\infty_{cc}(M)$. Applying the Poincar\'e-Lefschetz duality (see, Proposition \ref{prop: properties of c_pm on H_reg}) and substituting $H := -F$ we get
$c^+([M,\del M], F)\geq\inf_LF$  for all $F\in C^\infty_{cc}(M)$.
It follows that $\pqs(F)\geq\inf_LF$ for all $F\in C^\infty_{cc}(M)$ and thus, $L$ is a $\pqs$-heavy set.

Now, fix a Riemannian metric $g$ on $L$ all of whose contractible closed geodesics are constant and let $U_{\varepsilon}\cong\DD^*L$ be an $\varepsilon$-Weinstein's tubular neighborhood of $L$ w.r.t. the metric $g$. Consider a non-zero Hamiltonian $F\in\CL{S}_{HZ}^{\circ}(M)$, which is zero on $M\minus U_{\varepsilon}$ and  $F(q,p)=f(\|p\|^2)$ on $U_{\varepsilon}\cong\DD^*L$, where $f\in C^\infty_c([-1,1])$ with $\mathrm{supp}(f)\subset(-1,1)$, $f(r)\geq 0$ for all $r\in[-1,1]$ and $F|_L\equiv 0$. Then $\pqs(F)=\max_MF>0$ and $L$ is not a  $\pqs$-superheavy set.
\end{proof}

In the case of the cotangent bundle $(T^*X,\omega_{can})$ over a closed connected manifold $X$, one can use the results by Monzner, Vichery and Zapolsky \cite[Section 1.3.4]{M-V-Z} in order to get additional examples of $\pqs$-heavy sets. Indeed, since $\zeta_a\leq\pqs$ for any $a\in H^1(X;\RR)$, any $\zeta_a$-superheavy set is $\pqs$-heavy, where $\zeta_a:C_{cc}(T^*X)\to\RR$ is the Monzner-Vichery-Zapolsky partial quasi-state.

\subsection{The $C^0$-robustness of the Poisson bracket.}
Let $(M,L)=(\til{\DD^*X}, \til{X})$  be the relative one-point symplectic blow-up and $\qs:C_{cc}(M)\to\RR$ be the symplectic quasi-state as before. Equip the space $C_{cc}(M)$ with the uniform norm $\|F\|_{\infty}:=\max\limits_{x\in M}|F(x)|$. Following \cite{E-P-Z}, let $\Pi:C^{\infty}_{cc}(M)\times C^{\infty}_{cc}(M)\to\RR$ be a functional defined  by $\Pi(F_1,F_2) = |\qs(F_1+F_2)-\qs(F_1)-\qs(F_2)|$. It can be interpreted as a measure of Poisson non-commutativity of
functions $F_1$ and $F_2$. Since the arguments in \cite[Theoren 1.4]{E-P-Z} depend only on the formal properties of a symplectic quasi-state, we have the following
\begin{prop}
There exists a constant $C>0$ so that
\beq\label{formula:C0-robustness of PB}
\Pi(F_1,F_2) \leq C\cdot\sqrt{\|\{F_1,F_2\}\|_{\infty}}
\eeq
for all $F_1,F_2 \in C^{\infty}_{cc}(M)$.
\end{prop}
More generally, there is a lower bound on higher (iterated) Poisson brackets. Following \cite{E-P-R} we denote by $\CL{P}_N$ the set of Lie monomials in two variables involving $N$-times-iterated Poisson brackets (i.e. $\CL{P}_1$ consists of $\{F_1,F_2\}$, $\CL{P}_2$ of $\{\{F_1,F_2\},F_1\}$ and $\{\{F_1,F_2\},F_2\}$, and so forth).  For $F_1,F_2 \in C_{cc}^{\infty}(M)$ set
$$
Q_{N+1}(F_1,F_2) = \sum_{p \in \Cl{P}_N} \| p(F_1,F_2)\|_{\infty}.
$$
By \cite[Theorem 1.3]{E-P-R}, we have
\begin{prop}\label{prop: C0 robustness of PB}
There exist constants $C_N>0$ for any $N\in\NN$ so that
$
\Pi(F_1,F_2) \leq C_{N+1} \cdot Q_N(F_1,F_2)^{\frac{1}{N+1}}
$
for any $F_1,F_2 \in C_{cc}^{\infty}(M)$.
\end{prop}

\begin{cor}\label{cor: example of three lagrs for C0 rid of PB}
Let  $L_1, L_2, L_3$ be  three symplectic isotopies of $L$ by elements of $\Symp^0_c(M,\omega)$, such that $L_1\cap L_2\cap L_3=\varnothing$ and $L_i\pitchfork L_j$ for all $i, j$ and let
$F_1, F_2\in C_{cc}^{\infty}(M)$ be such that $F_i|_{L_i}\equiv 0$ for $i=1, 2$ and $(F_1+F_2)|_{L_3}\equiv 1$. Then $Q_N(F_1,F_2)\geq \(\frac{1}{C_{N+1}}\)^{N+1}$.
\end{cor}
\begin{proof}
Since $L$ is $\qs$-superheavy, it follows (by the symplectic invariance of $\qs$) that $L_1, L_2, L_3$ are $\qs$-superheavy as well. By \cite[Proposition 4.3]{E-P3}, we have $\Pi(F_1,F_2)=1$ and hence, the result follows from Proposition \ref{prop: C0 robustness of PB}.
\end{proof}

\subsection{Hamiltonian chords.}

Buhovsky-Entov-Polterovich \cite{Buh-E-P} used  Poisson brackets to define certain symplectic invariants of finite collections of sets in a symplectic manifold $(M,\omega)$.  We shall consider the invariant $pb_4(X_0,X_1,Y_0,Y_1)$ of quadruples of compact sets $X_0,X_1,Y_0,Y_1\subset M$ defined in \cite{Buh-E-P}.
It is defined by
$
pb_4(X_0,X_1,Y_0,Y_1) := \inf \|\{F_1,F_2\}\|_{\infty},
$
where the infimum is taken over the class
\begin{multline*} \CL{F}_4 (X_0, X_1, Y_0,Y_1) :=  \\ \{ (F_1,F_2)\in C_{cc}(M)^2\ |\
F_1|_{X_0} \leq 0,\ F_1|_{X_1} \geq 1,\ F_2|_{Y_0} \leq 0,\ F_2|_{Y_1} \geq
1\ \}.\end{multline*}
This class is non-empty whenever $ X_0 \cap X_1 = Y_0 \cap Y_1 =\varnothing$. If the latter condition is violated, $pb_4(X_0,X_1,Y_0,Y_1)$ is defined to be $+\infty$.

\noindent
The following result was proved in \cite{Buh-E-P}.
\begin{thm}
\label{thm-chords1}
Let $X_0,X_1,Y_0,Y_1 \subset M$ be a quadruple of compact sets
such that  $X_0 \cap X_1 = Y_0 \cap Y_1 =\emptyset$ and
$pb_4(X_0,X_1,Y_0,Y_1) = p > 0$. Let $G \in C^\infty_{cc} (M)$ be a
Hamiltonian function with $G|_{Y_0} \leq 0$ and $G|_{Y_1} \geq 1$
generating a Hamiltonian flow $\phi^t_G$. Then $\phi^T_G(x) \in X_1$ for some point $x \in X_0$ and some time moment $T \in [-1/p, 1/p]$.
\end{thm}

\medskip
The  curve $\{\phi^t_G(x) \}_{t \in [0;T]}$ is also called  a
\emph{\textbf{Hamiltonian chord}} of $\phi^t_G$ (or, for brevity, of the Hamiltonian
$G$) of time-length $|T|$ connecting $X_0$ and $X_1$. Hamiltonian chords joining two disjoint subsets (especially, Lagrangian submanifolds) of a symplectic manifold arise in several interesting contexts such as Arnold diffusion  or optimal control, see \cite{Buh-E-P} for the references.

From the estimate \eqref{formula:C0-robustness of PB} and \cite[Theorem $1.15 (ii)$]{Buh-E-P} we conclude
\begin{cor}\label{cor: pb4 low bound}
Let $M:=\til{\DD^*X}$  be the one-point symplectic blow-up of $\DD^*X$ relative to $X$ and $\qs:C_{cc}(M)\to\RR$ be the symplectic quasi-state as above. Let $X_0,X_1,Y_0,Y_1\subset \overset{\circ}{M}$ be a quadruple of
compact subsets such that
$X_0 \cap X_1 = Y_0 \cap Y_1 =\emptyset\;$
and $X_0 \cup Y_0, Y_0 \cup X_1, X_1 \cup Y_1, Y_1 \cup X_0$ are all $\qs$-superheavy. Then
\begin{equation}
\label{eq-Delta-1} pb_4(X_0,X_1,Y_0,Y_1) \geq \frac{1}{4C}\,,
\end{equation}
where $C$ is the constant that  appears in  \eqref{formula:C0-robustness of PB}.
\noindent In particular, there are Hamiltonian chords connecting $X_0$ and $X_1$.
\end{cor}

\begin{ex}\label{example: quadruple of superhevy}
Let $L:=\til{X}\subset M$ be the blown up zero section and let $L_0, L_1, L_2$ be its three symplectic isotopies of $L$ by elements of $\Symp^0_c(M,\omega)$, such that $L_0\cap L_1\cap L_2=\varnothing$ and $L_i\pitchfork L_j$ for all $i, j$. Note that all the Lagrangians $L_i$ are $\qs$-superheavy. Let $L_0\pitchfork L_1=\{x_1,\ldots, x_k\}$ and let $\DD_i\subset L_1$ be a small open disk around $x_i$, such that $L_2\cap\bigcup_{i=1}^k\cl{\DD_i}=\varnothing$.
Consider $X_0:=\cup_{i=1}^k\cl{\DD_i},  X_1:=L_2,  Y_0:=L_0,  Y_1:=L_1\minus(\cup_{i=1}^k\DD_i)$.
\end{ex}

\subsection{Symplectic approximation.}
The following problem was initiated in \cite{E-P-Z} and developed in \cite{E-P-R}, \cite{Buh-E-P}.\\
\newline
\noindent \textbf{Problem.\;}
\emph{Given a pair of functions on a symplectic manifold, what is its optimal uniform approximation by a pair of (almost) Poisson-commuting functions?}\\
\newline
Let $(M,L)=(\til{\DD^*X}, \til{X})$  and $\qs:C_{cc}(M)\to\RR$ be as before. Equip the space $\CL{F}:=C_{cc}^{\infty}(M)\times C_{cc}^{\infty}(M)$  with the uniform distance
$$
d((F_1,F_2),(G_1,G_2)):= \|F_1-G_1\|_{\infty} + \|F_2-G_2\|_{\infty}
$$
and consider the family of subsets $\CL{K}_s \subset \CL{F}$, $s \geq 0$, given by
$$
\CL{K}_s =\{(F_1,F_2)\in\CL{F}\;|\; \|\{F_1,F_2\}\|_{\infty} \leq s \}.
$$
The \emph{\textbf{profile function}} $\rho_{F_1,F_2}: [0;+\infty)\to\RR$ associated with a pair $(F_1,F_2) \in \CL{F}$ is defined by
$\rho_{F_1,F_2} (s) := d((F_1,F_2), \CL{K}_s)$ (see \cite{Buh-E-P}, \cite{E-P-R}).
The function $\rho_{F_1,F_2} (s)$ is non-increasing,  non-negative and   $\rho_{F_1,F_2} (\|\{F_1,F_2\}\|_{\infty})=0$. The value
$\rho_{F_1,F_2}(0)$ is responsible for the optimal uniform
approximation of $(F_1,F_2)$ by a pair of Poisson-commuting functions.

Take the quadruple $(X_0,X_1,Y_0,Y_1)$ of
compact subsets of $M$ from Example \ref{example: quadruple of superhevy}. By \cite[Theorem 1.4 (Dichotomy)]{Buh-E-P}, we have the following
\begin{cor}\label{cor: symplectic approximation}
For any pair $(F_1,F_2)\in\CL{F}$, such that at least one of the
functions $F_1$,$F_2$ has its range in $[0,1]$ and such that
$$
F_1|_{X_0} \leq 0,\ F_1|_{X_1} \geq 1,\ F_2|_{Y_0} \leq 0,\ F_2|_{Y_1} \geq 1,
$$
we have that the profile function $\rho_{F_1,F_2}$ is
continuous, $\rho_{F_1,F_2}(0) = 1/2$ and $$\frac{1/2-\rho_{F_1,F_2}(s)}{s}$$ is bounded away from $0$ and $+\infty$ for all sufficiently small $s>0$.  This gives a sharp rate, in terms of the power of $s$, of $\rho_{F_1,F_2}(s)$ near $s=0$.
\end{cor}

\end{document}